\documentclass[reqno]{amsart}
\usepackage{amssymb}
\usepackage{mathrsfs}
\usepackage{amsmath,amstext,amsfonts,verbatim}
\usepackage{color}
\usepackage{exscale}
\usepackage{relsize}
\usepackage{cite}

\newtheorem{theorem}{Theorem}[section] \newtheorem{lemma}[theorem]{Lemma} \newtheorem{proposition}[theorem]{Proposition} \newtheorem{corollary}[theorem]{Corollary} \theoremstyle{definition}          \theoremstyle{remark}  \numberwithin{equation}{section}

\title[]{ \textbf{Toeplitz operators on the Fock space via the Fourier transform}} \author[]{Shengkun Wu$^{1}$ and Dechao Zheng$^{2}$}

\address{$^{1}$ College of Mathematics and Statistics, Chongqing University, Chongqing, 401331, PR China} \email{shengkunwu@foxmail.com} \address{$^{2}$ Department of Mathematics, Vanderbilt University, Nashville, TN 37240 and Center of Mathematics, Chongqing University, Chongqing, 401331, PR China} \email{dechao.zheng@vanderbilt.edu}

 \thanks{The first author is supported by CSC201906050022. This work is partially supported by NFSC}
 \thanks{}
 \thanks{}

\date{}

\begin{document}

\begin{abstract} Insprited by Berger-Coburn theorems and their conjecture in \cite{Coburn1994}, we use the Fourier transform to decompose $ T_{g}$ as an infinite sum of Toeplitz operators with symbols which have compact support in the frequency domain.  As a consequence, we obtain a sufficient condition for $ T_{g}$ to  be bounded in terms of the Carleson measure conditions defined by the heat transform of the symbol $g$. Moreover the decomposition of  a Toeplitz operator leads us to get easily understanding that for a bounded function $g$, if its Berezin transform vanishes at infinity, then the Toeplitz operator $T_g$ is compact \cite{Eng} and the Toeplitz algebra generated by Toeplitz operators with symbols in $L^{\infty}$ is indeed generated by Toeplitz operators with symbols which on  uniformly continuous are ${\mathbb C}^n$ \cite{Bauer2012}.  \end{abstract} \maketitle

\section{Introduction }
Let $dv(x)$ denote the Lebesgue measure on $\mathbb{C}^n$. For any positive parameter $\alpha$,
let
$$d\lambda_{\alpha}(x)=(\frac{\alpha}{\pi})^n e^{-\alpha|x|^2}dv(x)$$
 denote the Gaussian measure on $\mathbb{C}^n$.
 Let $L^2(\mathbb{C}^n,d\lambda_{\alpha})$ be the Hilbert space consisting of all square integrable functions with respect to the Gaussian measure.
The Fock space $F_{\alpha}^2$ is a subspace of $L^2(\mathbb{C}^n,d\lambda_{\alpha})$ which consists all holomorphic functions $f$ on $\mathbb{C}^n$.
Let $P$ be the projection from $L^2(\mathbb{C}^n,d\lambda_{\alpha})$ onto the Fock space.
The inner product of two functions $f,g\in F_{\alpha}^2$ is given by
$$\langle f,g\rangle=\int_{\mathbb{C}^n}f\overline{g}d\lambda_{\alpha} .$$

For $z,w\in \mathbb{C}^n$, on one hand  the inner product $\langle z,w\rangle$ is defined by
$$\langle z,w\rangle=z_1\overline{w_1}+z_2\overline{w_2}+\dots+z_n\overline{w_n}$$
where $z=(z_1, z_2,\dots,z_n)$ and $w=(w_1, w_2,\dots,w_n)$.
  On the other hand, $z$ and $w$ can be viewed as two vectors in $\mathbb{R}^{2n}.$ We
 denote the inner product of $z$ and $w$ in $\mathbb{R}^{2n}$ as
$$z \cdot w=\Re z \cdot \Re w+\Im z \cdot \Im w =\Re z_1\Re w_1+\cdots+\Re z_n\Re w_n+\Im z_1\Im w_1+\dots+\Im z_n\Im w_n$$
where
$$\Re z=(\Re z_1,\Re z_2,\dots,\Re z_n), \text{\quad and \quad}\Im z=(\Im z_1,\Im z_2,\dots,\Im z_n),$$
$ \Re z_i$ and $\Im z_i$ denote the real part and the imagery part of the complex number $z_i$ respectively.
Thus $$z \cdot w= \Re \langle z,w\rangle.$$
 For a multi-index $\sigma =(\sigma_1,\dots,\sigma_n)\in \mathbb{Z}_{+}^n$, the order $|\sigma |$ of the multi-index $\sigma$ is defined by
$$|\sigma| =\sigma_1+\dots+\sigma_n .$$
The monomial $z^{\sigma}$ is defined by
$$z^{\sigma}=z_1^{\sigma_1}z_2^{\sigma_2}\cdots z_n^{\sigma_n}.$$
 Similarly, we define the differential operators $\partial^{\sigma}_{\Re z}$ and $\partial^{\sigma}_{\Im z}$  by
 $$\partial^{\sigma}_{\Re z}=\partial^{\sigma_1}_{\Re z_1}\cdots\partial^{\sigma_n}_{\Re z_n}\text{ and }\partial^{\sigma}_{\Im z}=\partial^{\sigma_1}_{\Im z_1}\cdots\partial^{\sigma_n}_{\Im z_n}.$$

The Fock space is a reproducing kernel Hilbert space and for any $w,z\in \mathbb{C}^n$ with the reproducing kernel $K_w(z)$   given by
$$K_{w}(z)=e^{\alpha \langle z,w\rangle}.$$
The normalized reproducing kernel $k_{w}(z)$ is giving by
$$k_w(z)=e^{\alpha \langle z,w\rangle-\frac{\alpha}{2}|w|^2}.$$
Since the linear expansion of $\{k_w\}$ is dense in the Fock space,
for a measurable function $g$ on $\mathbb{C}^n$, if for any $w\in \mathbb{C}^n$ we have $gk_w\in L^2(\mathbb{C}^n,d\lambda_{\alpha})$, then we define the Toeplitz operator $T_g$ with symbol $g$ on the linear combinations of $\{k_w\}$ by
$$T_g\sum_{j=1}^nc_jk_{w_j}=P\Big(g\sum_{j=1}^nc_jk_{w_j}\Big).$$

A natural problem is  what   necessary and sufficient conditions are
for $ T_{g}$ to  be bounded. In \cite{Coburn1994},  Berger and Coburn use  some trace formula to obtain   some necessary conditions for $ T_{g}$ to  be bounded and  use the Bargmann transform and Calderon-Vaillancourt pseudo-differential estimates to obtain some sufficient conditions  for $ T_{g}$ to  be bounded in terms of the heat transform of the symbol $g$. These conditions lead them to make   conjecture on the problem \cite{Coburn1994}.

Insprited by Berger-Coburn theorems and their conjecture in \cite{Coburn1994}, we use the Fourier transform to decompose $ T_{g}$ as an infinite sum of Toeplitz operators with symbols which have compact support in the frequency domain.

To state our results precisely we need to introduce notations. For any positive parameter $t$, the heat transform of $g$ is defined by
$$\mathcal{H}_{t}g(z)=(\frac{1}{t\pi})^n \int_{\mathbb{C}^n} g(w) \mathrm{e}^{-\frac{|z-w|^{2}}{t}} dv(w), \text{ for any } z \in \mathbb{C}^n.$$
In fact, if $t=\frac{1}{\alpha}$, a simple calculation gives that the heat transform of $g$ is equal to the Berezin transform of the Toeplitz operator $T_{g}$:
$$\mathcal{H}_{\frac{1}{\alpha}}g(z)=\langle T_{g}k_z, k_z\rangle .$$
For two functions $f$ and $g$ in $L^2(\mathbb{C}^n, dv)$, the convolution $f\ast g$ of $f$ and $g$ is defined by
$$f\ast g(z)=\int_{\mathbb{C}^n}f(z-w)g(w)dv(w)$$
for any $z\in\mathbb{C}^n$.

For any $y\in \mathbb{C}^{n}$, $\tau_{y}$ is a translation operator:
$$\tau_{y}f(x)=f(x-y).$$
For a Lebesgue integrable function $f$ on ${\mathbb C}^n$, the Fourier transform $ \mathcal{F}f$ of $f$ is defined by
$$ \mathcal{F}f(z)=\int_{\mathbb{R}^{2n}}f(w)e^{-2\pi i z\cdot w}dv(w).$$
Let $\gamma_t(z)$ be $2n$-dimension heat kernel $(\frac{1}{t\pi})^n \mathrm{e}^{-\frac{|z|^{2}}{t}}$ and let
$a_{t}(z)$ be the  Gaussian function $e^{-\pi^2t|z|^2}$.  Then the heat transform of $g$ is given by
$$\mathcal{H}_{t}g(z)=g\ast\gamma_t(z)$$
and
$$\mathcal{F}^{-1}\gamma_{t}(z)=\mathcal{F}\gamma_{t}(z)=a_{t}(z)$$
as the Gaussian function $e^{-\pi |z|^2}$ is fixed by the Fourier transform.

In the case of $\alpha=\frac{1}{2}$, in  \cite{Coburn1994}, Berger and Coburn proved that if $\mathcal{H}_{t}g$ is bounded for some $t\in (0,\frac{1}{2\alpha})$, then $T_g$ is bounded. In fact, they showed that if
$$\partial^{a}_{\Re y}\partial^{b}_{\Im y}\mathcal{H}_{\frac{1}{2\alpha}}g(y)$$
is bounded for any $a,b\in \mathbb{Z}^n$ with $|a|+|b|\leq 2n+1$, then $T_g$ is bounded.
In this paper, we will improve the above conditions replacing the boundedness  by the Carleson measure condition on the Fock space.

Let $d\mu$ be a positive measure on $\mathbb{C}^n$. For any $r>0$ and $x\in\mathbb{C}^n$, let $B(x,r)$ denote the ball in $\mathbb{C}^n$ with center $x$ and radius $r$. If there is  a positive constant $C_r$ such that
$$\sup_{x\in\mathbb{C}^n}\mu(B(x,r))<C_{r},$$
then we say that $d\mu$ is a  Carleson  measure on the Fock space. It is shown in  \cite[Theorem 3.29]{Zhu}  that $d\mu$ is a  Carleson  measure on the Fock space if and only if
$$\sup_{w\in\mathbb{C}^n}\int_{\mathbb{C}^n} \mathrm{e}^{-\frac{p\alpha}{2}|z-w|^{2}} d\mu(z) < \infty,$$
for some $0<p<\infty$.
Thus, for a nonnegative function $f$ on $\mathbb{C}^n$, ${\mathcal H}_{\frac{2}{\alpha}}(f)\in L^{\infty}$ if and only if $f(y)dv(y)$ is a  Carleson  measure on the Fock space. Moreover,
for any $r>0$, there is a $C_r$ such that
\begin{equation}\label{equation 1.1}
\frac{1}{C_r}\|\mathcal{H}_{\frac{2}{\alpha}}(f)\|_{\infty}\leq \sup_{x}\int_{B(x,r)}f(y)dv(y)\leq C_r \|\mathcal{H}_{\frac{2}{\alpha}}(f)\|_{\infty}.
\end{equation}

We will obtain a sufficient condition for a Toeplitz operator to be bounded and show that how the heat transformation of the symbol of a Toeplitz operator is related to the operator itself. In fact, a Toeplitz operator $T_g$ can be represented by the weighted integral of "translates" of  $T_{\mathcal{H}_{\frac{1}{2\alpha}+t}g}$.
To do so, we introduce the partition of unity.

Let $Q_0=\{\xi +i\eta : \xi,\eta \in (-\frac{1}{2},\frac{1}{2}]^n\}$ be a rectangle in $\mathbb{C}^n$ and $x_0=(0,\dots,0 )$ be the center of $Q_0$.  Let $\Gamma$ be the lattice ${\mathbb Z}^{2n}$, i.e.,
$$ \Gamma=\{ z=(z_1, z_2, \cdots z_n)\in {\mathbb C}^n:\Re z_1, \Im z_1,\dots, \Re z_n,\Im z_n \in \mathbb{Z}\}.$$
For any $x\in \Gamma$, let $Q_x$ denote the rectangle with center $x$, which is a translate of $Q_0$, i.e.,
$$Q_x=Q_0 +x.$$
Then the collection $\{Q_x:x\in \Gamma\}$ of rectangles tiles  $\mathbb{C}^n$.


Let $\phi$ be a smooth function such that
 $$\phi (z)=\left \{\begin{array}{cc} 1  & ~~ z\in Q_0\\
0  &   ~~z\notin 2Q_0  .  \end{array} \right.$$
where $2Q_0= (-1,1]^{2n}$.

For any $x\in \Gamma$ and $z\in \mathbb{C}^n$ , let $\phi_{x}(z)=\phi(z-x)$. Since
\begin{equation}\label{equation 2.1}
\sum_{x\in \Gamma}\frac{1}{(1+|\Re x|+|\Im x|)^{2n+1}}\leq \sum_{x\in \Gamma}\frac{1}{(1+|x|)^{2n+1}}<\infty,
\end{equation}
  the series
$\sum_{x\in\Gamma } \phi_{x}$ of smooth functions  converges uniformly on compact subsets of ${\mathbb C}^n$  and satisfies
$$\frac{1}{c}\leq |\sum_{x\in\Gamma } \phi_{x}(z)|\leq c$$
for $z$ in  ${\mathbb C}^n$ and for some positive constant $c$. Letting
$$ \varphi_{x}=\frac{\phi_{x}}{\sum_{x\in\Gamma } \phi_{x}} ,$$
 we obtain
$$\varphi_{x}(z+x)=\varphi_{x_0}(z),$$ and $$\sum_{x\in\Gamma}\varphi_{x}=1.$$

The following theorem is our main result.

\begin{theorem}\label{Theorem 1.2}
Let g be a measurable function on $\mathbb{C}^n$ such that $gk_w\in L^2(\mathbb{C}^n,d\lambda_{\alpha})$ for any $w\in \mathbb{C}^n$. For each $x$ in $\Gamma ={\mathbb Z}^{2n}$, let $$g_x=(\mathcal{H}_{\frac{1}{2\alpha}}g)\ast\mathcal{F}(\varphi_{x} a_{\frac{1}{2\alpha}}^{-1}).$$ If for any $a,b\in \mathbb{Z}_{+}^n$ with $|a|+|b|\leq 2n+1$, $|\partial^{a}_{\Re y}\partial^{b}_{\Im y}\mathcal{H}_{\frac{1}{2\alpha}}g(y)|dv(y)$ is a  Carleson measure on the Fock space, then the Toeplitz operator $T_g$ is decomposed as a sum of Toeplitz operators with symbols which have compact support    in the frequency domain:
$$T_g=\sum_{x\in \Gamma}T_{g_x}$$
 and for any $t\geq 0,$
$$T_{g_x}=\int_{\mathbb{C}^n}\mathcal{F}(\varphi_{x} a_{\frac{1}{2\alpha}+t}^{-1})(y)T_{\tau_{y}\mathcal{H}_{\frac{1}{2\alpha}+t}g}dv(y)
,$$
where the summation and the integral are both convergent in the operator norm topology and hence $T_g$ is bounded. Moreover
$$\|T_g\|\leq C_{n,\alpha}\sum_{|a|+|b|\leq 2n+1}\|\mathcal{H}_{\frac{2}{\alpha}}(|\partial^{a}_{\Re y}\partial^{b}_{\Im y}\mathcal{H}_{\frac{1}{2\alpha}}g(y)|) \|_{\infty},$$
where $C_{n,\alpha}$ is a positive constant depending only on $n$ and $\alpha$.
\end{theorem}
{\bf Remarks.} Let us make some remarks about the above theorem. First as
$$g_x=(\mathcal{H}_{\frac{1}{2\alpha}}g)\ast\mathcal{F}(\varphi_{x} a_{\frac{1}{2\alpha}}^{-1}),$$
its Fourier inverse transform
$${\mathcal F}^{-1}(g_x)(z)=[{\mathcal F}^{-1}(\mathcal{H}_{\frac{1}{2\alpha}}g)](z)\varphi_{x}(z) a_{\frac{1}{2\alpha}}^{-1}(z)$$
has support in $2Q_x=2Q_0+x$ and hence $g_x$ has compact support in the frequency domain.

Since
\begin{equation}\label{equation 1.2}
\sum_{|a|+|b|\leq 2n+1}\|\mathcal{H}_{\frac{2}{\alpha}}(|\partial^{a}_{\Re y}\partial^{b}_{\Im y}\mathcal{H}_{\frac{1}{2\alpha}}g(y)|) \|_{\infty}\lesssim
\sum_{|a|+|b|\leq 2n+1}\|\partial^{a}_{\Re y}\partial^{b}_{\Im y}\mathcal{H}_{\frac{1}{2\alpha}}g(y) \|_{\infty}\lesssim\|\mathcal{H}_tg\|_{\infty},
\end{equation}
where $t\in (0,\frac{1}{2\alpha})$, the above theorem   improves  \cite[Theorem 6.18]{Zhu}.

The decomposition of a Toeplitz operator in the above theorem is useful for us to understand two  results on Toeplitz operators and Toeplitz algebras in  \cite{Bauer2012} and \cite{Eng}.
If $\mathcal{A}$ is a set of functions that satisfies the condition in Theorem \ref{Theorem 1.2}, let $\mathfrak{T}(\mathcal{A})$ denote the $C^*-$algebra  generated by $\{T_u: u\in \mathcal{A}\}$ and $\mathfrak{T}^{1}(\mathcal{A})$ denote the closed space which is generated by $\{T_u: u\in \mathcal{A}\}$ in the norm topology.

\begin{corollary}\label{corollary 1.2}
Let g be a measurable function on $\mathbb{C}^n$ such that $gk_w\in L^2(\mathbb{C}^n,d\lambda_{\alpha})$ for any $w\in \mathbb{C}^n$. Suppose that for any $a,b\in \mathbb{Z}_{+}^n$ with $|a|+|b|\leq 2n+1$, $|\partial^{a}_{\Re y}\partial^{b}_{\Im y}\mathcal{H}_{\frac{1}{2\alpha}}g(y)|dv(y)$ is a Carleson measure on the Fock space. If for a fixed $t\geq 0,$
$\tau_y\mathcal{H}_{\frac{1}{2\alpha}+t}g$ is in $\mathcal{A}$
for any $y\in\mathbb{C}^n$, then $T_g\in \mathfrak{T}^1(\mathcal{A}).$
\end{corollary}
\begin{proof}
Suppose that for a fixed $t\geq 0$,  $\tau_y\mathcal{H}_{\frac{1}{2\alpha}+t}g$ is in $\mathcal{A}$ for any $y\in\mathbb{C}^n$. Thus $T_{\tau_y\mathcal{H}_{\frac{1}{2\alpha}+t}g}$ is in $\mathfrak{T}^1(\mathcal{A}).$
By Theorem \ref{Theorem 1.2}, we have
$$T_{g_x}=\int_{\mathbb{C}^n}\mathcal{F}(\varphi_{x} a_{\frac{1}{2\alpha}+t}^{-1})(y)T_{\tau_{y}\mathcal{H}_{\frac{1}{2\alpha}+t}g}dv(y).$$
Since the above integral converges to $T_{g_x}$ in the operator norm topology, $T_{g_x}$ is in $\mathfrak{T}^1(\mathcal{A}).$
By Theorem \ref{Theorem 1.2} again, we have
$$T_g=\sum_{x\in \Gamma}T_{g_x}$$
to conclude that $T_g$ is in $\mathfrak{T}^1(\mathcal{A}).$

\end{proof}

Corollary 1.2 generalizes \cite[Proposition 7.4.1]{robert2} in the case $p = 2$ and immediately leads to the two results on the Toeplitz algebras and compact Toeplitz operators in  \cite{Bauer2012} and \cite{Eng}.

Let $C_{buc}(\mathbb{C}^n)$ be the set of bounded uniformly continuous function (\cite[page 1345]{Bauer2012}). In \cite{Bauer2012}, on Toeplitz algebras, Bauer and Isralowitz showed that $$\mathfrak{T}(L^{\infty})=\mathfrak{T}(C_{buc}(\mathbb{C}^n)).$$
 We will show how one gets easily
\begin{equation}\label{equation 4.1}
\mathfrak{T}^1(L^{\infty})=\mathfrak{T}^1(C_{buc}(\mathbb{C}^n))\text{ and }\mathfrak{T}(L^{\infty})=\mathfrak{T}(C_{buc}(\mathbb{C}^n)).
\end{equation}
To do so,  we notice that for any $g\in L^{\infty}$ and $y$,  $\tau_y\mathcal{H}_{\frac{1}{2\alpha}}g$ is in $ C_{buc}(\mathbb{C}^n)$. Thus
$T_g$ is in $\mathfrak{T}^1(C_{buc}(\mathbb{C}^n))$, and so the above corollary gives (\ref{equation 4.1}). By \cite{xia}, we have $$\mathfrak{T}^1(L^{\infty})=\mathfrak{T}(L^{\infty}).$$
Thus, we actually get
$$\mathfrak{T}^1(L^{\infty})=\mathfrak{T}^1(C_{buc}(\mathbb{C}^n))=\mathfrak{T}(L^{\infty})=\mathfrak{T}(C_{buc}(\mathbb{C}^n)).$$

On compact Toeplitz operators,   in \cite{Eng}, Engli\v s showed that   for $g\in L^{\infty}$, if the Berezin transform of $T_g$ vanishes at infinity, that is,
$$\lim_{|z|\rightarrow \infty }\mathcal{H}_{\frac{1}{\alpha}}g(z)=0,$$
then the Toeplitz operator $T_g$ is compact.

To get the above result, we observe that if for $g\in L^{\infty}$,
 $$\lim_{|z|\rightarrow \infty}\mathcal{H}_{\frac{1}{\alpha}}g(z)=0,$$
 then for each $y$ in ${\mathbb C}^n$,
 $$\lim_{|z|\rightarrow \infty}\tau_y \mathcal{H}_{\frac{1}{\alpha}}g(z)=\lim_{|z|\rightarrow \infty} \mathcal{H}_{\frac{1}{\alpha}}g(z-y)=\lim_{|w|\rightarrow \infty} \mathcal{H}_{\frac{1}{\alpha}}g(w)=0.$$
 Let $$\mathcal{A}=\{g\in L^{\infty} (\mathbb{C}^n):\lim_{|z|\rightarrow \infty}g(z)=0 \}.$$
 Then  $\tau_y \mathcal{H}_{\frac{1}{\alpha}}g$ is in $\mathcal{A}$ for any $y$. So we have that  $T_g\in\mathfrak{T}^1(\mathcal{A}).$ Noting that it is not difficult to show that $\mathfrak{T}^1(\mathcal{A})\subset \mathcal{K}$, we conclude that $T_{g}$ is compact.

This paper is organized as follows. In Section 2, using the partition of unity to decompose the symbol $g$ of the Toeplitz operator $T_g$, we will show that each part in the decomposition is bounded and has compact support in frequency domain. In Section 3, we will obtain the norm estimation of the Toeplitz operator with symbol equal to each part of the decomposition to establish  a norm estimation of $T_g$. In Section 4, we will apply our decomposition theory for a Toeplitz operator to estimate the Schatten $p$-norm of the product of two Toeplitz operators.

\section{decomposition}
In this section, first we will establish some decomposition of  a symbol $g$ of a Toeplitz operator  $T_{g}$. Even if  $g$ is a measurable function on $\mathbb{C}^n$ such that $gk_w\in L^2(\mathbb{C}^n,d\lambda_{\alpha})$ for any $w\in \mathbb{C}^n$, the Fourier transform ${\mathcal F}(f)$ may be   a tempered distribution. So we need to recall some facts on tempered distributions and the Fourier transform. Next using the decomposition of the symbol we will obtain a decomposition of the symbol $g$ of the Toeplitz operator $T_g.$

We now introduce the Schwartz space $\mathscr{S}\left(\mathbb{R}^{2n}\right)$. A smooth complex-valued function $f$ on $\mathbb{R}^{2n}$ is called a Schwartz function if for every pair of multi-indices $\alpha$ and $\beta$ there exists a positive constant $C_{\alpha,\beta}$ such that
$$\rho_{\alpha,\beta}(f)=\sup_{x\in\mathbb{R}^{2n}}|\partial^{\alpha}(x^{\beta}f(x))|=C_{\alpha,\beta}<\infty.$$
The set of all Schwartz functions on $\mathbb{R}^{2n}$ is called the Schwartz space and denoted by $\mathscr{S}\left(\mathbb{R}^{2n}\right)$. The Schwartz space $\mathscr{S}\left(\mathbb{R}^{2n}\right)$ is a locally convex topological vector space equipped with the family of seminorms $\rho_{\alpha,\beta}$.

Elements of the dual space $\mathscr{S}'\left(\mathbb{R}^{2n}\right)$ of the Schwartz space are called tempered distributions. A function $g$ is said to be a tempered distribution if for any $f$ in the Schwartz space, the pair
$(f,g)$ defined by
$$(f,g)=\int_{\mathbb{C}^n} f(x) g(x) dv(x)$$
gives a continuous linear functional on the Schwartz space.

Next, we need to recall some facts about the tempered distribution.
If $f,h\in\mathscr{S}\left(\mathbb{R}^{2n}\right)$ and $G\in \mathscr{S}'\left(\mathbb{R}^{2n}\right)$.
The Fourier transformation can be extend to the dual of the Schwartz space $\mathscr{S}\left(\mathbb{R}^{2n}\right)$ such that $\mathcal{F}(G)$ is a tempered distribution with
$$(\mathcal{F}(G),f)=(G,\mathcal{F}(f)).$$
$Gh$ is a tempered distribution given by
$$(Gh,f)=(G,hf),$$
and $G\ast h$ is a tempered distribution such that
$$(G\ast h,f)=(G,\tilde{h}\ast f),$$
where $\tilde{h}(x)=h(-x)$ for any $x\in \mathbb{R}^{2n}$.
$\tilde {G}$ is also a tempered distribution such that
$$(\tilde{G},h)=(G,\tilde{h}).$$
For any $y\in \mathbb{R}^{2n}$, $\tau_{y}$ is a translation operator on $\mathscr{S}\left(\mathbb{R}^{2n}\right)$ such that $\tau_{y}f(x)=f(x-y)$.
$\tau_{y}$ can be extended on $\mathscr{S}'\left(\mathbb{R}^{2n}\right)$ such that
$$ (\tau_{y}G,f)=(G,\tau^{-1}_{y}f).$$

The following lemma says that a Carleson measure on the Fock space induces a tempered distribution.
\begin{lemma}\label{Lemma 2.1}
Let $h$ be a positive function on $\mathbb{C}^n$ such that $h(y)dv(y)$ is a Carleson measure on the Fock space, then $h$ is a tempered distribution. Moreover, if $f$ is a Schwartz function, then $h\ast f$ is in $ L^{\infty}(\mathbb{C}^n).$
\end{lemma}
\begin{proof} Since $h(y)dv(y)$ is a Carleson measure on the Fock space,
 we have
$$\sup_{x\in\mathbb{C}^n}\int_{B(x,2)} |h(y)|dv(y)=C_2  < \infty.$$
To show that $h$ is a tempered distribution, by \cite[Proposition 2.3.4]{Grafakos}, it is sufficient to show
$$|\int_{\mathbb{C}^n}h(y)f(y)dv(y)| \leq C_3 \sup_{y}(1+|\Re y|+|\Im y|)^{2n+1}|f(y)| $$
for all $f$ in the Schwartz space $\mathscr{S}\left(\mathbb{R}^{2n}\right)$ and for some positive constant $C_3.$
As we pointed out in the introduction that $\{Q_x\}_{x\in \Gamma}$ tiles ${\mathbb C}^n$, for any $f$   in the Schwartz space $\mathscr{S}\left(\mathbb{R}^{2n}\right)$, we have
\begin{align*}
&|\int_{\mathbb{C}^n}h(y)f(y)dv(y)|=|\int_{\cup_{x\in\Gamma}Q_x}h(y)f(y)dv(y)|\\
\leq& \sum_{x\in \Gamma }|\int_{Q_x}h(y)f(y)dv(y)|\\
\leq& \sum_{x\in \Gamma}|\int_{Q_x}\frac{1}{(1+|\Re y|+|\Im y|)^{2n+1}}h(y)
(1+|\Re y|+|\Im y|)^{2n+1}f(y)dv(y)|.
\end{align*}
Since for any $y\in Q_x$ with $x\in \Gamma$, there is a positive constant $C$ such that
$$\frac{1}{(1+|\Re y|+|\Im y|)^{2n+1}} \leq C \frac{1}{(1+|\Re x|+|\Im x|)^{2n+1}}.$$
Thus
\begin{align*}
&|\int_{\mathbb{C}^n}h(y)f(y)dv(y)|\\
\leq &\sum_{x\in \Gamma}|\int_{Q_x}\frac{1}{(1+|\Re y|+|\Im y|)^{2n+1}}h(y)
(1+|\Re y|+|\Im y|)^{2n+1}f(y)dv(y)|\\
\leq&\sum_{x\in \Gamma } C \frac{1}{(1+|\Re x|+|\Im x|)^{2n+1}}\int_{Q_x}|h(y)
(1+|\Re y|+|\Im y|)^{2n+1}f(y)|dv(y)\\
\leq&\sum_{x\in \Gamma } C \frac{1}{(1+|\Re x|+|\Im x|)^{2n+1}}\int_{Q_x}h(y)dv(y)
\sup_{y}(1+|\Re y|+|\Im y|)^{2n+1}|f(y)|\\
\leq&\sum_{x\in \Gamma } C \frac{1}{(1+|\Re x|+|\Im x|)^{2n+1}}\int_{B(x,2)}h(y)dv(y)
\sup_{y}(1+|\Re y|+|\Im y|)^{2n+1}|f(y)| \\
\leq &C_3  \sup_{y}(1+|\Re y|+|\Im y|)^{2n+1}|f(y)|
\end{align*}
for some positive constant $C_3$.
The last inequality follows from
(\ref{equation 2.1}).
Thus  $h$ is a tempered distribution.

In fact, the above argument gives that for any $z\in \mathbb{C}^n$
$$
|\int_{\mathbb{C}^n}h(z-y)f(y)dv(y)|\leq C_3\sup_{y}(1+|\Re y|+|\Im y|)^{2n+1}|f(y)|.
$$
This implies that $h\ast f$ is in $L^{\infty}(\mathbb{C}^n)$.
\end{proof}

\begin{lemma}\label{Lemma 2.2}
Let g be a measurable function on $\mathbb{C}^n$ such that $gk_w\in L^2(\mathbb{C}^n,d\lambda_{\alpha})$ for any $w\in \mathbb{C}^n$. If $|\mathcal{H}_{\frac{1}{2\alpha}}g(y)|dv(y)$ is a Carleson measure on the Fock space, then $(\mathcal{H}_{\frac{1}{2\alpha}}g)\ast\mathcal{F}(\varphi_{x} a_{\frac{1}{2\alpha}}^{-1})\in L^{\infty}(\mathbb{C}^n)$ for any $x\in \Gamma$ and
$$\mathcal{H}_{\frac{1}{\alpha}}g(z)=\sum_{x\in \Gamma}\mathcal{H}_{\frac{1}{\alpha}}\Big[ (\mathcal{H}_{\frac{1}{2\alpha}}g)\ast\mathcal{F}(\varphi_{x} a_{\frac{1}{2\alpha}}^{-1})\Big](z),$$
for any $z\in \mathbb{C}^n.$
\end{lemma}
\begin{proof}

Since for each $x\in \Gamma$, $\varphi_{x} a_{\frac{1}{2\alpha}}^{-1}$ is a smooth function with compact support,
 $\varphi_{x} a_{\frac{1}{2\alpha}}^{-1}$ is in the Schwartz space. Thus $\mathcal{F}(\varphi_{x} a_{\frac{1}{2\alpha}}^{-1})$ is a Schwartz function as the Fourier transform $\mathcal F$ maps the Schwartz space onto the Schwartz space.   So  Lemma \ref{Lemma 2.1} gives that
$$(\mathcal{H}_{\frac{1}{2\alpha}}g)\ast\mathcal{F}(\varphi_{x} a_{\frac{1}{2\alpha}}^{-1})\in L^{\infty}({\mathbb C}^n)$$
as $|\mathcal{H}_{\frac{1}{2\alpha}}g(y)|dv(y)$ is a Carleson measure on the Fock space.

Since $gk_w\in L^2(\mathbb{C}^n,d\lambda_{\alpha})$ for any $w\in \mathbb{C}^n$, we have
$$\mathcal{H}_{{\frac{1}{\alpha}}}g(z)=\langle g k_z,k_z\rangle= g \ast \gamma_{\frac{1}{\alpha}}(z)$$
is finite for any $z\in \mathbb{C}^{n}$. Then by the Fubini theorem, we have
$$\mathcal{H}_{\frac{1}{\alpha}}g= (g\ast \gamma_{\frac{1}{2\alpha}})\ast \gamma_{\frac{1}{2\alpha}}=(\mathcal{H}_{\frac{1}{2\alpha}}g)\ast \gamma_{\frac{1}{2\alpha}}=(\mathcal{H}_{\frac{1}{2\alpha}}g)\ast \mathcal{F}( a_{\frac{1}{2\alpha}}).$$
By Lemma \ref{Lemma 2.1}, $\mathcal{H}_{\frac{1}{2\alpha}}g$ is a tempered distribution, thus $\tau_y \widetilde{\mathcal{H}_{\frac{1}{2\alpha}}g}$ is also a tempered distribution.
So we have
\begin{align*}
\mathcal{H}_{\frac{1}{\alpha}}g(z)&=\int_{\mathbb{C}^n}\mathcal{H}_{\frac{1}{2\alpha}}g(z-x) \mathcal{F}(a_{\frac{1}{2\alpha}})(x)dv(x)\\
&=(\tau_z \widetilde{\mathcal{H}_{\frac{1}{2\alpha}}g},\mathcal{F}(a_{\frac{1}{2\alpha}}))\\
&=(\mathcal{F}\big[ \tau_z \widetilde{\mathcal{H}_{\frac{1}{2\alpha}}g}\big],a_{\frac{1}{2\alpha}}).
\end{align*}
On the other hand, using properties of the convolution, the Fourier transform and the heat transform, we have
\begin{align*}
\mathcal{H}_{\frac{1}{\alpha}}\Big[(\mathcal{H}_{\frac{1}{2\alpha}}g)\ast\mathcal{F}(\varphi_{x} a_{\frac{1}{2\alpha}}^{-1})\Big](z)
&=(\mathcal{H}_{\frac{1}{2\alpha}}g)\ast\mathcal{F}(\varphi_{x} a_{\frac{1}{2\alpha}}^{-1})\ast\gamma_{\frac{1}{\alpha}}(z)\\
&=(\mathcal{H}_{\frac{1}{2\alpha}}g)\ast[\mathcal{F}(\varphi_{x} a_{\frac{1}{2\alpha}}^{-1})\ast\gamma_{\frac{1}{\alpha}}](z)\\
&=(\mathcal{H}_{\frac{1}{2\alpha}}g)\ast[\mathcal{F}\mathcal{F}^{-1}[\mathcal{F}(\varphi_{x} a_{\frac{1}{2\alpha}}^{-1})\ast\gamma_{\frac{1}{\alpha}}]](z)\\
&=(\mathcal{H}_{\frac{1}{2\alpha}}g)\ast\mathcal{F}\big[(\mathcal{F}^{-1}\gamma_{\frac{1}{\alpha}})\varphi_{x} a_{\frac{1}{2\alpha}}^{-1}\big](z)\\
&=(\mathcal{H}_{\frac{1}{2\alpha}}g)\ast\mathcal{F}\big[a_{\frac{1}{\alpha}}\varphi_{x} a_{\frac{1}{2\alpha}}^{-1}\big](z)\\
&=(\mathcal{H}_{\frac{1}{2\alpha}}g)\ast \mathcal{F}\big[\varphi_{x} a_{\frac{1}{2\alpha}}\big](z)\\
&=(\tau_z \widetilde{\mathcal{H}_{\frac{1}{2\alpha}}g},\mathcal{F}\big[\varphi_{x} a_{\frac{1}{2\alpha}}\big])\\
&=(\mathcal{F}\big[\tau_z \widetilde{\mathcal{H}_{\frac{1}{2\alpha}}g}\big], \varphi_{x} a_{\frac{1}{2\alpha}}).
\end{align*}
Since $\mathcal{F}\big[\tau_z \widetilde{\mathcal{H}_{\frac{1}{2\alpha}}g}\big]$ is a tempered distribution and $\sum_{x\in \Gamma}\varphi_{x} a_{\frac{1}{2\alpha}}$ converges to $a_{\frac{1}{2\alpha}}$ in the Schwartz space, we have
$$\sum_{x\in \Gamma}(\mathcal{F}\big[\tau_z \widetilde{\mathcal{H}_{\frac{1}{2\alpha}}g}\big], \varphi_{x}a_{\frac{1}{2\alpha}})=(\mathcal{F}\big[ \tau_z \widetilde{\mathcal{H}_{\frac{1}{2\alpha}}g}\big],a_{\frac{1}{2\alpha}}).$$
Thus we conclude
\begin{align*}
\mathcal{H}_{\frac{1}{\alpha}}g(z) &=(\mathcal{F}\big[ \tau_z \widetilde{\mathcal{H}_{\frac{1}{2\alpha}}g}\big],a_{\frac{1}{2\alpha}})\\
 & =\sum_{x\in \Gamma}(\mathcal{F}\big[\tau_z \widetilde{\mathcal{H}_{\frac{1}{2\alpha}}g}\big], \varphi_{x}a_{\frac{1}{2\alpha}})\\
 &=\sum_{x\in \Gamma}\mathcal{H}_{\frac{1}{\alpha}}\Big[(\mathcal{H}_{\frac{1}{2\alpha}}g)\ast\mathcal{F}(\varphi_{x} a_{\frac{1}{2\alpha}}^{-1})\Big](z),
\end{align*}
to complete the proof.
\end{proof}

For $x$ in $\Gamma$, let $$g_x=(\mathcal{H}_{\frac{1}{2\alpha}}g)\ast\mathcal{F}(\varphi_{x} a_{\frac{1}{2\alpha}}^{-1}).$$
The above lemma tells us that the summation of the Berezin transform  of $T_{g_x}$ is equal to the Berezin transform  of $T_{g}$. In fact, we will show that $\sum_{x\in \Gamma}T_{g_x}$ converges to $T_g$ in operator norm topology in the last section.

For any $z\in\mathbb{C}^n$, we define a unitary operator $W_z$ on the Fock space such that
$$W_zf(w)=f(w-z)k_z(w)$$
for any $f$ in the Fock space. We have  $$W_z^{*}=W_{-z}.$$
Let
$$b_x(w)=e^{2\pi i w\cdot x}$$
for any $w,x\in\mathbb{C}^n.$

\begin{lemma}\label{Lemma 2.3}
Let g be a measurable function on $\mathbb{C}^n$ such that $gk_w\in L^2(\mathbb{C}^n,d\lambda_{\alpha})$ for any $w\in \mathbb{C}^n$. If $|\mathcal{H}_{\frac{1}{2\alpha}}g(y)|dv(y)$ is a Carleson measure on the Fock space, then
$$W_{\frac{-i\pi x}{2\alpha}}T_{g_x}W_{\frac{-i\pi x}{2\alpha}}=
T_{(b_{x}\mathcal{H}_{\frac{1}{2\alpha}}g)\ast\mathcal{F}[a_{\frac{1}{2\alpha}}^{-1}\varphi_{x_0}]}.$$
Further, we have
$$
T_{g_x}W_{-\frac{i\pi x}{\alpha}}= T_{(b_{x}\mathcal{H}_{\frac{1}{2\alpha}}g)\ast \tau_{-\frac{i\pi x}{2\alpha} }\mathcal{F}[a_{\frac{1}{2\alpha}}^{-1}\varphi_{x_0}]} \text{\quad and \quad} W_{-\frac{i\pi x}{\alpha}}T_{g_x}= T_{(b_{x}\mathcal{H}_{\frac{1}{2\alpha}}g)\ast \tau_{\frac{i\pi x}{2\alpha} }\mathcal{F}[a_{\frac{1}{2\alpha}}^{-1}\varphi_{x_0}]}.
$$
\end{lemma}
\begin{proof}
Since $a_{\frac{1}{2\alpha}}^{-1}\varphi_{x_0}$ is a smooth function with compact support, we have
$$\mathcal{F}[a_{\frac{1}{2\alpha}}^{-1}\varphi_{x_0}]\in \mathscr{S}\left(\mathbb{R}^{2n}\right).$$
Since $|b_{x}(z)|=1$, by Lemma \ref{Lemma 2.1} , we have
$$ (b_{x}\mathcal{H}_{\frac{1}{2\alpha}}g)\ast\mathcal{F}[a_{\frac{1}{2\alpha}}^{-1}\varphi_{x_0}]\in L^{\infty}({\mathbb C}^n).$$
By Lemma \ref{Lemma 2.2}, we have $g_{x}\in L^{\infty}({\mathbb C}^n)$ and hence $T_{g_x}$ is a bounded operator.
Because the Berezin transform is injective, to get
$$W_{\frac{-i\pi x}{2\alpha}}T_{g_x}W_{\frac{-i\pi x}{2\alpha}}=
T_{(b_{x}\mathcal{H}_{\frac{1}{2\alpha}}g)\ast\mathcal{F}[a_{\frac{1}{2\alpha}}^{-1}\varphi_{x_0}]},$$
   we only need to show that
$$\langle  [W_{\frac{-i\pi x}{2\alpha}}T_{g_x}W_{\frac{-i\pi x}{2\alpha}}]k_z, k_z\rangle =\langle T_{(b_{x}\mathcal{H}_{\frac{1}{2\alpha}}g)\ast
\mathcal{F}[a_{\frac{1}{2\alpha}}^{-1}\varphi_{x_0}]}k_z, k_z\rangle.$$
On one hand, since $\mathcal{H}_{\frac{1}{\alpha}}(g)$ is the Berezin transform of the Toeplitz operator $T_g$, we have
\begin{equation*}
\begin{split}
\langle T_{(b_{x}\mathcal{H}_{\frac{1}{2\alpha}}g)\ast
\mathcal{F}[a_{\frac{1}{2\alpha}}^{-1}\varphi_{x_0}]}k_z, k_z\rangle
=&\mathcal{H}_{\frac{1}{\alpha}}\Big[(b_{x}\mathcal{H}_{\frac{1}{2\alpha}}g)\ast\mathcal{F}[a_{\frac{1}{2\alpha}}^{-1}\varphi_{x_0}]\Big]\\
=& (b_{x}\mathcal{H}_{\frac{1}{2\alpha}}g)\ast\mathcal{F}[a_{\frac{1}{2\alpha}}^{-1}\varphi_{x_0}]\ast\gamma_{\frac{1}{\alpha}}\\
=&(b_{x}\mathcal{H}_{\frac{1}{2\alpha}}g)\ast\mathcal{F}[a_{\frac{1}{2\alpha}}^{-1}\varphi_{x_0}]\ast\mathcal{F}a_{\frac{1}{\alpha}}\\
=&(b_{x}\mathcal{H}_{\frac{1}{2\alpha}}g)\ast\mathcal{F}[a_{\frac{1}{2\alpha}}\varphi_{x_0}].\\
\end{split}
\end{equation*}
By properties of the Fourier transform and noting that $\mathcal{F}\tau_{-x}=b_x\mathcal{F}$, we have
\begin{equation*}
\begin{split}
\mathcal{F}[a_{\frac{1}{2\alpha}}\varphi_{x_0}]
=&\mathcal{F}[a_{\frac{1}{2\alpha}}(\tau_{-x}\varphi_{x})]\\
=&\mathcal{F}[\tau_{-x}(\varphi_{x}a^{-1}_{\frac{1}{2\alpha}}) \times \tau_{-x}(a_{\frac{1}{2\alpha}} \tau_{x}a_{\frac{1}{2\alpha}})]\\
=&\mathcal{F}[\tau_{-x}(\varphi_{x}a^{-1}_{\frac{1}{2\alpha}})]  \ast \mathcal{F}[\tau_{-x}(a_{\frac{1}{2\alpha}} \tau_{x}a_{\frac{1}{2\alpha}})]\\
=&b_x\mathcal{F}[(\varphi_{x}a^{-1}_{\frac{1}{2\alpha}})]  \ast b_x\mathcal{F}[(a_{\frac{1}{2\alpha}} \tau_{x}a_{\frac{1}{2\alpha}})].
\end{split}
\end{equation*}
For any $w\in \mathbb{C}^n$, a simple calculation gives
\begin{align*}
\mathcal{F}[a_{\frac{1}{2\alpha}} \tau_{x}a_{\frac{1}{2\alpha}}](w)
&=\int_{\mathbb{C}^n}e^{\frac{-\pi^2}{2\alpha}|z-x|^2}e^{\frac{-\pi^2}{2\alpha}|z|^2}e^{-2\pi iz\cdot w}dv(z)\\
&=e^{\frac{-\pi^2}{4\alpha}|x|^2} \int_{\mathbb{C}^n}e^{\frac{-\pi^2}{\alpha}|z-\frac{x}{2}|^2}e^{-2\pi iz\cdot w}dv(z)\\
&=e^{\frac{-\pi^2}{4\alpha}|x|^2} e^{-\pi ix\cdot w} \int_{\mathbb{C}^n}e^{\frac{-\pi^2}{\alpha}|z|^2}e^{-2\pi iz\cdot w}dv(z)\\
&=e^{\frac{-\pi^2}{4\alpha}|x|^2} e^{-\pi ix\cdot w} (\frac{\alpha}{\pi})^n e^{-\alpha |w|^2}.
\end{align*}
Thus, we have
\begin{align*}
&\langle T_{(b_{x}\mathcal{H}_{\frac{1}{2\alpha}}g)\ast
\mathcal{F}[a_{\frac{1}{2\alpha}}^{-1}\varphi_{x_0}]}k_z, k_z\rangle\\
=&(b_{x}\mathcal{H}_{\frac{1}{2\alpha}}g)\ast [b_x(\mathcal{F}\varphi_{x}a^{-1}_{\frac{1}{2\alpha}})]\ast b_x\mathcal{F}[a_{\frac{1}{2\alpha}}\tau_{x}(a_{\frac{1}{2\alpha}})]\\
=& \int_{\mathbb{C}^n}\int_{\mathbb{C}^n}b_{x}(y)\mathcal{H}_{\frac{1}{2\alpha}}g(y) b_x(w-y)(\mathcal{F}\varphi_{x}a^{-1}_{\frac{1}{2\alpha}})(w-y)dv(y) b_x(z-w)\mathcal{F}[a_{\frac{1}{2\alpha}}\tau_{x}(a_{\frac{1}{2\alpha}})](z-w)dv(w)\\
=& \int_{\mathbb{C}^n}\int_{\mathbb{C}^n}\mathcal{H}_{\frac{1}{2\alpha}}g(y) (\mathcal{F}\varphi_{x}a^{-1}_{\frac{1}{2\alpha}})(w-y)dv(y) e^{2\pi iz\cdot x }\mathcal{F}[a_{\frac{1}{2\alpha}}\tau_{x}(a_{\frac{1}{2\alpha}})](z-w)dv(w)\\
=& \int_{\mathbb{C}^n}g_x(w) e^{\frac{-\pi^2}{4\alpha}|x|^2} e^{\pi ix\cdot (z+w)} (\frac{\alpha}{\pi})^n e^{-\alpha |z-w|^2}dv(w).
\end{align*}
On the other hand, we have
\begin{align*}
&\langle  [W_{\frac{-i\pi x}{2\alpha}}T_{g_x}W_{\frac{-i\pi x}{2\alpha}}]k_z, k_z\rangle \\
=&\langle T_{g_x}W_{\frac{-i\pi x}{2\alpha}}k_z, W_{\frac{i\pi x}{2\alpha}}k_z \rangle\\
=&\int_{\mathbb{C}^n} g_x k_z(w+\frac{i\pi x}{2\alpha})k_{\frac{-i\pi x}{2\alpha}}(w) \overline{k_z(w-\frac{i\pi x}{2\alpha})k_{\frac{i\pi x}{2\alpha}}(w)} d\lambda_{\alpha}(w)\\
=&\int_{\mathbb{C}^n} g_x
e^{\alpha \langle w+\frac{i\pi x}{2\alpha},z\rangle-\frac{\alpha |z|^2}{2}-\alpha \langle w,\frac{i\pi x}{2\alpha}\rangle-\frac{\pi^2|x|^2}{8\alpha}}
e^{\alpha \langle z ,w-\frac{i\pi x}{2\alpha}\rangle-\frac{\alpha |z|^2}{2}+\alpha \langle\frac{i\pi x}{2\alpha},w\rangle-\frac{\alpha\pi^2|x|^2}{8}}
 d\lambda_{\alpha}(w)\\
 =&\int_{\mathbb{C}^n} g_x(w) e^{\frac{-\pi^2}{4\alpha}|x|^2}
e^{\alpha \langle \frac{i\pi x}{2\alpha},z\rangle-\alpha \langle w,\frac{i\pi x}{2\alpha}\rangle
+\alpha \langle z ,-\frac{i\pi x}{2\alpha}\rangle+\alpha \langle\frac{i\pi x}{2\alpha},w\rangle}
 (\frac{\alpha}{\pi})^n e^{-\alpha |z-w|^2}dv(w).
\end{align*}
Since
\begin{align*}
&\alpha \langle \frac{i\pi x}{2\alpha},z\rangle-\alpha \langle w,\frac{i\pi x}{2\alpha}\rangle
+\alpha \langle z ,-\frac{i\pi x}{2\alpha}\rangle+\alpha \langle\frac{i\pi x}{2\alpha},w\rangle\\
=&\frac{i\pi }{2}\langle x,z\rangle+\frac{i\pi }{2}\langle w,x\rangle
+\frac{i\pi }{2}\langle z,x\rangle+\frac{i\pi }{2}\langle x,w\rangle\\
=&i\pi  \Re \langle x,z+w\rangle\\
=&i\pi x\cdot(z+w),
\end{align*}
we have
\begin{align*}
&\langle  [W_{\frac{-i\pi x}{2\alpha}}T_{g_x}W_{\frac{-i\pi x}{2\alpha}}]k_z, k_z\rangle \\
=&\int_{\mathbb{C}^n}g_x(w) e^{\frac{-\pi^2}{4\alpha}|x|^2} e^{\pi ix\cdot (z+w)} (\frac{\alpha}{\pi})^n e^{-\alpha |z-w|^2}dv(w)\\
=&\langle T_{(b_{x}\mathcal{H}_{\frac{1}{2\alpha}}g)\ast
\mathcal{F}[a_{\frac{1}{2\alpha}}^{-1}\varphi_{x_0}]}k_z, k_z\rangle.
\end{align*}
Further, we have
\begin{align*}
T_{g_x}W_{-\frac{i\pi x}{\alpha}}= & W_{\frac{i\pi x}{2\alpha}}W_{-\frac{i\pi x}{2\alpha}}T_{g_x}W_{-\frac{i\pi x}{2\alpha}}W_{-\frac{i\pi x}{2\alpha}}\\
=& W_{\frac{i\pi x}{2\alpha}}T_{(b_{x}\mathcal{H}_{\frac{1}{2\alpha}}g)\ast \mathcal{F}[a_{\frac{1}{2\alpha}}^{-1}\varphi_{x_0}]}W_{-\frac{i\pi x}{2\alpha}}\\
=& T_{(b_{x}\mathcal{H}_{\frac{1}{2\alpha}}g)\ast \tau_{-\frac{i\pi x}{2\alpha} }\mathcal{F}[a_{\frac{1}{2\alpha}}^{-1}\varphi_{x_0}]} .
\end{align*}
Similarly, we have
\begin{align*}
W_{-\frac{i\pi x}{\alpha}}T_{g_x}
= & W_{-\frac{i\pi x}{2\alpha}}W_{-\frac{i\pi x}{2\alpha}}T_{g_x}W_{-\frac{i\pi x}{2\alpha}}W_{\frac{i\pi x}{2\alpha}}\\
=& W_{-\frac{i\pi x}{2\alpha}}T_{(b_{x}\mathcal{H}_{\frac{1}{2\alpha}}g)\ast \mathcal{F}[a_{\frac{1}{2\alpha}}^{-1}\varphi_{x_0}]}W_{\frac{i\pi x}{2\alpha}}\\
=& T_{(b_{x}\mathcal{H}_{\frac{1}{2\alpha}}g)\ast \tau_{\frac{i\pi x}{2\alpha} }\mathcal{F}[a_{\frac{1}{2\alpha}}^{-1}\varphi_{x_0}]}
\end{align*}
to complete the proof.
\end{proof}
\section{estimation}
In this section, we will present the proof of our main theorem. In Section 2, we have obtained that $T_g$ is decomposed as a sum $\sum_{x\in \Gamma}T_{g_x}$of the Toeplitz operators via the Berezin transform (Lemma \ref{Lemma 2.2}).  To show that $T_g=\sum_{x\in \Gamma}T_{g_x}$, we need to estimate the norm of a Toeplitz operator on the Fock space and show that the series of operators converges in operator norm topology. The following lemma is a generalization of \cite[Lemma 4.9]{Bauer2020}.

\begin{lemma}\label{lemma 3.0} Let T be a densely defined operator on the Fock space and the linear span of reproducing kernels is contained in the domain of $T$ and $T^*$. If
$$[\sup_{z}\int_{\mathbb{C}^n}|\langle Tk_z, k_w \rangle|dv(w)][ \sup_{w}\int_{\mathbb{C}^n}|\langle Tk_z, k_w \rangle|dv(z)]<\infty,$$
then $T$ is bounded and
$$\|T\|\leq \big([(\frac{\alpha}{\pi})^n\sup_{z}\int_{\mathbb{C}^n}|\langle Tk_z, k_w \rangle|dv(w)]
[\sup_{w}\int_{\mathbb{C}^n}|\langle Tk_z, k_w \rangle|dv(z)]\big)^{1/2}.$$
\end{lemma}
\begin{proof}
Let $f$ and $h$ be in the linear span of reproducing kernels,  we have
\begin{align*}
\langle Tf,h\rangle&=\int_{\mathbb{C}^n}(Tf)(w)\overline{h(w)}d\lambda_{\alpha}(w)\\
&=(\frac{\alpha}{\pi})^n\int_{\mathbb{C}^n}\langle Tf,k_w\rangle\langle k_w,h\rangle dv(w)\\
&=(\frac{\alpha}{\pi})^n\int_{\mathbb{C}^n}\langle f,T^*k_w\rangle\langle k_w,h\rangle dv(w)\\
&=(\frac{\alpha}{\pi})^{2n}\int_{\mathbb{C}^n}\int_{\mathbb{C}^n} \langle f,k_z\rangle
\langle k_z,T^*k_w\rangle dv(z)\langle k_w,h\rangle dv(w)\\
&=(\frac{\alpha}{\pi})^{2n}\int_{\mathbb{C}^n}\int_{\mathbb{C}^n} \langle f,k_z\rangle
\langle Tk_z, k_w\rangle dv(z)\langle k_w,h\rangle dv(w)
\end{align*}
Thus applying the Cauchy-Schwarz inequality gives
\begin{align*}
|\langle Tf,h\rangle|
&\leq (\frac{\alpha}{\pi})^{2n}\int_{\mathbb{C}^n}\int_{\mathbb{C}^n}|\langle Tk_z, k_w \rangle| |\langle f,k_z\rangle|dv(z) |\langle k_w,h\rangle| dv(w)\\
&\leq (\frac{\alpha}{\pi})^{2n}\big[\int_{\mathbb{C}^n}\big(\int_{\mathbb{C}^n}|\langle Tk_z, k_w \rangle| |\langle f,k_z\rangle|dv(z)\big)^2dv(w)\big]^{1/2} \big[\int_{\mathbb{C}^n}|\langle k_w,h\rangle|^2 dv(w)\big]^{1/2}\\
&=(\frac{\alpha}{\pi})^{2n}\big[\int_{\mathbb{C}^n}\big(\int_{\mathbb{C}^n}|\langle Tk_z, k_w \rangle| |\langle f,k_z\rangle|dv(z)\big)^2dv(w)\big]^{1/2} \big[\int_{\mathbb{C}^n}|h(w)|^2 d\lambda_{\alpha}(w)\big]^{1/2}.
\end{align*}
Since the Cauchy-Schwarz inequality gives
\begin{align*}
&\int_{\mathbb{C}^n}\big(\int_{\mathbb{C}^n}|\langle Tk_z, k_w \rangle| |\langle f,k_z\rangle|dv(z)\big)^2dv(w)\\
=&\int_{\mathbb{C}^n}\big(\int_{\mathbb{C}^n}|\langle Tk_z, k_w \rangle|^{1/2} |\langle Tk_z, k_w \rangle|^{1/2} |\langle f,k_z\rangle|dv(z)\big)^2dv(w)\\
\leq &\int_{\mathbb{C}^n}\int_{\mathbb{C}^n}|\langle Tk_z, k_w\rangle|dv(z) \int_{\mathbb{C}^n} |\langle Tk_z, k_w\rangle|
|\langle f,k_z\rangle|^2 dv(z)dv(w)\\
\leq& [\sup_{w}\int_{\mathbb{C}^n}|\langle Tk_z, k_w \rangle|dv(z)][ \int_{\mathbb{C}^n}\int_{\mathbb{C}^n} |\langle Tk_z, k_w \rangle|
|\langle f,k_z\rangle|^2 dv(z)dv(w)]\\
\leq& [\sup_{w}\int_{\mathbb{C}^n}|\langle Tk_z, k_w \rangle|dv(z)][ \sup_{z}\int_{\mathbb{C}^n} |\langle Tk_z, k_w \rangle|dv(w) ] \int_{\mathbb{C}^n}
|\langle f,k_z\rangle|^2 dv(z)\\
=&[\sup_{w}\int_{\mathbb{C}^n}|\langle Tk_z, k_w \rangle|dv(z)][ \sup_{z}\int_{\mathbb{C}^n} |\langle Tk_z, k_w \rangle|dv(w) ] \int_{\mathbb{C}^n}
|f(z)|^2d\lambda_{\alpha}(z),
\end{align*}
we obtain
\begin{align*}
&|\langle Tf,h\rangle|\\
\leq & (\frac{\alpha}{\pi})^{2n}\big[[\sup_{w}\int_{\mathbb{C}^n}|\langle Tk_z, k_w \rangle|dv(z)][ \sup_{z}\int_{\mathbb{C}^n} |\langle Tk_z, k_w \rangle|dv(w)]\\
&\int_{\mathbb{C}^n}|f(z)|^2 d\lambda_{\alpha}(z) \int_{\mathbb{C}^n}|h(w)|^2 d\lambda_{\alpha}(w)\big]^{1/2}\\
=& (\frac{\alpha}{\pi})^n\big[[\sup_{w}\int_{\mathbb{C}^n}|\langle Tk_z, k_w \rangle|dv(z) ][\sup_{z}\int_{\mathbb{C}^n} |\langle Tk_z, k_w \rangle|dv(w)]\big]^{1/2}
 \|f\|_{F_{\alpha}^2} \|h\|_{F_{\alpha}^2}.
\end{align*}
This completes the proof.
\end{proof}

We introduce two-variable Berezin transform.  Let $g$ be a measurable function on $\mathbb{C}^n$ such that $gk_w\in L^2(\mathbb{C}^n,d\lambda_{\alpha})$ for any $w\in \mathbb{C}^n$. The two-variable Berezin transform $\tilde{g}(z, w)$ of $g$ is defined by
$$\tilde{g}(z, w)=\langle gk_z, k_w \rangle$$
for $z,~ w$ in $\mathbb{C}^n.$ By Lemma \ref{lemma 3.0}, we immediately obtain the following proposition which will give an estimation of the norm of the Toeplitz operator $T_g$ in terms of the two-variable Berezin transform $\tilde{g}(z, w)$ of $g$.

\begin{proposition}\label{proposition 3.1}
Let g be a measurable function on $\mathbb{C}^n$ such that $gk_w\in L^2(\mathbb{C}^n,d\lambda_{\alpha})$ for any $w\in \mathbb{C}^n$.
If
$$[\sup_{z}\int_{\mathbb{C}^n}|\tilde{g}(z,w)|dv(w)][ \sup_{w}\int_{\mathbb{C}^n}|\tilde{g}(z,w)|dv(z)]<\infty,$$
then $T_g$ is bounded and
$$\|T_g\|\leq \big([(\frac{\alpha}{\pi})^n\sup_{z}\int_{\mathbb{C}^n}|\tilde{g}(z,w)|dv(w)]
[\sup_{w}\int_{\mathbb{C}^n}|\tilde{g}(z,w)|dv(z)]\big)^{1/2}.$$
\end{proposition}

For any $x\in\Gamma$, Lemma \ref{Lemma 2.2} tells us that  $g_x=(\mathcal{H}_{\frac{1}{2\alpha}}g)\ast\mathcal{F}(\varphi_{x} a_{\frac{1}{2\alpha}}^{-1})$ is in $L^\infty$. Thus $T_{g_x}$ is a bounded operator.  Next, we will obtain a better estimation on the norm of $T_{g_x}$ than $\|g_{x}\|_{\infty}$ to guarantee that the series $\sum_{x\in \Gamma}T_{g_x}$ converges in the operator norm topology.

\begin{proposition}\label{proposition 3.2}
Let g be a measurable function on $\mathbb{C}^n$ such that $gk_w\in L^2(\mathbb{C}^n,d\lambda_{\alpha})$ for any $w\in \mathbb{C}^n$.
If for any $a,b\in \mathbb{Z}_{+}^n$ with $|a|+|b|\leq 2n+1$, $|\partial^{a}_{\Re y}\partial^{b}_{\Im y}\mathcal{H}_{\frac{1}{2\alpha}}g(y)|dv(y)$ is a Carleson measure on the Fock space. Then there is a constant $C_{n,\alpha}$ such that
$$\|T_{g_x}\| \leq
C_{n,\alpha}\frac{\sum_{|a|+|b|\leq 2n+1}\|\mathcal{H}_{\frac{2}{\alpha}}(|\partial^{a}_{\Re y}\partial^{b}_{\Im y}\mathcal{H}_{\frac{1}{2\alpha}}g(y)|)\|_{{\infty}} }{(1+|\Re x|+|\Im x|)^{2n+1}}.$$
\end{proposition}
\begin{proof}
Since $W_{\frac{-i\pi x}{2\alpha}}$ is an unitary operator and $g_x\in L^{\infty}$, we have
$$\|T_{g_x}\|=\|W_{\frac{-i\pi x}{2\alpha}}T_{g_x}W_{\frac{-i\pi x}{2\alpha}}\| .$$
By Lemma \ref{Lemma 2.3}, we have
$$\|T_{g_x}\|=\|T_{(b_{x}\mathcal{H}_{\frac{1}{2\alpha}}g)\ast\mathcal{F}[a_{\frac{1}{2\alpha}}^{-1}\varphi_{x_0}]}\|.$$
 Let $\psi_{x}=(b_{x}\mathcal{H}_{\frac{1}{2\alpha}}g)\ast\mathcal{F}[a_{\frac{1}{2\alpha}}^{-1}\varphi_{x_0}].$ Proposition \ref{proposition 3.1} gives
\begin{align*}
\|T_{g_x} \|&=\|T_{\psi_x}\|\\
&\lesssim \big([\sup_{z}\int_{\mathbb{C}^n}|\widetilde{\psi_x}(z, w)|dv(w)]
[\sup_{w}\int_{\mathbb{C}^n}|\widetilde{\psi_x}(z, w)|dv(z)]\big)^{1/2}.
\end{align*}

Let $M_{z,w}(y)= k_{z}(y)\overline{k_w(y)}(\frac{\alpha}{\pi})^n e^{-\alpha|y|^2} $.  For any $a,b\in \mathbb{Z}_{+}^n$ with $|a|+|b|\leq 2n+1$, we have
\begin{align*}
&|(\Re x)^a(\Im x)^b \widetilde{\psi_x}(z, w)|\\
=&|(\Re x)^a(\Im x)^b\langle(b_{x}\mathcal{H}_{\frac{1}{2\alpha}}g)\ast\mathcal{F}[a_{\frac{1}{2\alpha}}^{-1}\varphi_{x_0}] k_z, k_w \rangle|\\
=&|\int_{\mathbb{C}^n}(\Re x)^a(\Im x)^b(b_{x}\mathcal{H}_{\frac{1}{2\alpha}}g)\ast\mathcal{F}[a_{\frac{1}{2\alpha}}^{-1}\varphi_{x_0}](y)
k_z(y) \overline{k_w(y)}(\frac{\alpha}{\pi})^n e^{-\alpha|y|^2}dv(y)|\\
=&|\int_{\mathbb{C}^n}(\Re x)^a(\Im x)^b(b_{x}\mathcal{H}_{\frac{1}{2\alpha}}g)\ast\mathcal{F}[a_{\frac{1}{2\alpha}}^{-1}\varphi_{x_0}](y)
M_{z,w}(y)dv(y)|\\
=&|\int_{\mathbb{C}^n}(\Re x)^a(\Im x)^b\int_{\mathbb{C}^n} e^{2\pi i x \cdot (y-\eta)}
\mathcal{H}_{\frac{1}{2\alpha}}g(y-\eta)\mathcal{F}[a_{\frac{1}{2\alpha}}^{-1}\varphi_{x_0}](\eta)dv(\eta)
M_{z,w}(y)dv(y)|\\
=&|\int_{\mathbb{C}^n}\int_{\mathbb{C}^n} (\Re x)^a(\Im x)^b e^{-2\pi i x \cdot\eta}
\mathcal{H}_{\frac{1}{2\alpha}}g(y-\eta)\mathcal{F}[a_{\frac{1}{2\alpha}}^{-1}\varphi_{x_0}](\eta)dv(\eta)
e^{2\pi i x \cdot y}M_{z,w}(y)dv(y)|
\end{align*}
Since for any $a,b\in \mathbb{Z}_{+}^n$ with $|a|+|b|\leq 2n+1$, $\partial^{a}_{\Re \eta}\partial^{b}_{\Im \eta}\mathcal{H}_{\frac{1}{2\alpha}}g(\eta)dv(\eta)$ is a Carleson measure on the Fock space and $\mathcal{F}[a_{\frac{1}{2\alpha}}^{-1}\varphi_{x_0}]$ is in the Schwartz space,
 by the proof of Lemma \ref{Lemma 2.1} we have
$\partial^{a}_{\Re \eta}\partial^{b}_{\Im \eta}[\mathcal{H}_{\frac{1}{2\alpha}}g(y-\eta)\mathcal{F}[a_{\frac{1}{2\alpha}}^{-1}\varphi_{x_0}](\eta)]$
is integrable with respect to $\eta$. By the properties of Fourier transform \cite[page 12]{Duoandikoetxea}, we have
\begin{equation}\label{equation 3.1}
\begin{split}
&\int_{\mathbb{C}^n}(\Re x)^a(\Im x)^b e^{-2\pi i x \cdot \eta}
\mathcal{H}_{\frac{1}{2\alpha}}g(y-\eta)\mathcal{F}[a_{\frac{1}{2\alpha}}^{-1}\varphi_{x_0}](\eta)dv(\eta)\\
=&(\frac{-1}{2\pi i})^{|a|+|b|}\int_{\mathbb{C}^n} e^{-2\pi i x \cdot \eta}\partial^{a}_{\Re \eta}\partial^{b}_{\Im \eta}
[\mathcal{H}_{\frac{1}{2\alpha}}g(y-\eta)\mathcal{F}[a_{\frac{1}{2\alpha}}^{-1}\varphi_{x_0}](\eta)]dv(\eta)\\
=&\int_{\mathbb{C}^n} e^{-2\pi i x \cdot \eta}\sum_{a'\leq a,b'\leq b}C_{a,b,a',b'}
[\partial^{a'}_{\Re \eta}\partial^{b'}_{\Im \eta}\mathcal{H}_{\frac{1}{2\alpha}}g(y-\eta)]
\partial^{a-a'}_{\Re \eta}\partial^{b-b'}_{\Im \eta}\mathcal{F}[a_{\frac{1}{2\alpha}}^{-1}\varphi_{x_0}(\eta)]dv(\eta),\\
\end{split}
\end{equation}
where $\{C_{a,b,a',b'}\}$ are constants and $a'\leq a,b'\leq b$ means for any $j\leq n$ we have $a_j'\leq a_j,b_j'\leq b_j.$ Since
$|M_{z,w}(y)|=(\frac{\alpha}{\pi})^n e^{-\frac{\alpha}{2}|y-z|^2-\frac{\alpha}{2}|y-w|^2}$and
$$\int_{\mathbb{C}^n} |M_{z,w}(y)|dv(w)\thickapprox e^{-\frac{\alpha}{2}|y-z|^2},$$
we have
\begin{align*}
&|(\Re x)^a(\Im x)^b|\sup_{z}\int_{\mathbb{C}^n}|\widetilde{\psi_{x}}(z,w)|dv(w)\\
=&\sup_{z}\int_{\mathbb{C}^n}|\int_{\mathbb{C}^n}\int_{\mathbb{C}^n} e^{-2\pi i x \cdot \eta}\sum_{a'\leq a,b'\leq b}C_{a,b,a',b'}
[\partial^{a'}_{\Re \eta}\partial^{b'}_{\Im \eta}\mathcal{H}_{\frac{1}{2\alpha}}g(y-\eta)]\\
&\times \partial^{a-a'}_{\Re \eta}\partial^{b-b'}_{\Im \eta}\mathcal{F}[a_{\frac{1}{2\alpha}}^{-1}\varphi_{x_0}(\eta)]dv(\eta)e^{2\pi i x \cdot y}M_{z,w}(y)dv(y)|dv(w)\\
\leq& C\sum_{a'\leq a,b'\leq b}\sup_{z}\int_{\mathbb{C}^n}\int_{\mathbb{C}^n}\int_{\mathbb{C}^n}
|\partial^{a'}_{\Re \eta}\partial^{b'}_{\Im \eta}\mathcal{H}_{\frac{1}{2\alpha}}g(y-\eta)|\\
\end{align*}
\begin{align*}
\times & |\partial^{a-a'}_{\Re \eta}\partial^{b-b'}_{\Im \eta}\mathcal{F}[a_{\frac{1}{2\alpha}}^{-1}\varphi_{x_0}(\eta)]|dv(\eta)|M_{z,w}(y)|dv(y)dv(w)\\
\leq& C\sum_{a'\leq a,b'\leq b}\sup_{z}\int_{\mathbb{C}^n}\int_{\mathbb{C}^n}
|\partial^{a'}_{\Re \eta}\partial^{b'}_{\Im \eta}\mathcal{H}_{\frac{1}{2\alpha}}g(y-\eta)|\\
\times & \int_{\mathbb{C}^n}|M_{z,w}(y)|dv(w)dv(y)|\partial^{a-a'}_{\Re \eta}\partial^{b-b'}_{\Im \eta}
\mathcal{F}[a_{\frac{1}{2\alpha}}^{-1}\varphi_{x_0}(\eta)]|dv(\eta)\\
\leq&C\sum_{a'\leq a,b'\leq b}\int_{\mathbb{C}^n}\sup_{z}\int_{\mathbb{C}^n}
|\partial^{a'}_{\Re \eta}\partial^{b'}_{\Im \eta}\mathcal{H}_{\frac{1}{2\alpha}}g(y-\eta)|e^{-\frac{\alpha}{2}|y-z|^2}dv(y)|\partial^{a-a'}_{\Re \eta}\partial^{b-b'}_{\Im \eta}
\mathcal{F}[a_{\frac{1}{2\alpha}}^{-1}\varphi_{x_0}(\eta)]|dv(\eta).
\end{align*}
Since for any $\eta$,
\begin{equation*}
\begin{split}
&\sup_{z}\int_{\mathbb{C}^n}|\partial^{a'}_{\Re \eta}\partial^{b'}_{\Im \eta}\mathcal{H}_{\frac{1}{2\alpha}}g(y-\eta)|e^{-\frac{\alpha}{2}|y-z|^2}dv(y)\\
=&\sup_{z}\int_{\mathbb{C}^n}|\partial^{a'}_{\Re y}\partial^{b'}_{\Im y}\mathcal{H}_{\frac{1}{2\alpha}}g(y-\eta)|e^{-\frac{\alpha}{2}|y-z|^2}dv(y)\\
=&\sup_{z}\int_{\mathbb{C}^n}|\partial^{a'}_{\Re y}\partial^{b'}_{\Im y}\mathcal{H}_{\frac{1}{2\alpha}}g(y)|e^{-\frac{\alpha}{2}|y+\eta-z|^2}dv(y)\\
=&(\frac{\pi}{\alpha})^n\|\mathcal{H}_{\frac{2}{\alpha}}(|\partial^{a'}_{\Re y}\partial^{b'}_{\Im y}\mathcal{H}_{\frac{1}{2\alpha}}g|) \|_{\infty},
\end{split}
\end{equation*}
we obtain
\begin{align*}
&|(\Re x)^a(\Im x)^b|\sup_{z}\int_{\mathbb{C}^n}|\widetilde{\psi_{x}}(z,w)|dv(w)\\
\leq&C\sum_{a'\leq a,b'\leq b}(\frac{\pi}{\alpha})^n\|\mathcal{H}_{\frac{2}{\alpha}}(|\partial^{a'}_{\Re y}\partial^{b'}_{\Im y}\mathcal{H}_{\frac{1}{2\alpha}}g(y)|) \|_{\infty}\int_{\mathbb{C}^n}|\partial^{a-a'}_{\Re \eta}\partial^{b-b'}_{\Im \eta}
\mathcal{F}[a_{\frac{1}{2\alpha}}^{-1}\varphi_{x_0}(\eta)]|dv(\eta)\\
\leq&C'\sum_{a'\leq a,b'\leq b}\|\mathcal{H}_{\frac{2}{\alpha}}(|\partial^{a'}_{\Re y}\partial^{b'}_{\Im y}\mathcal{H}_{\frac{1}{2\alpha}}g|) \|_{\infty},
\end{align*}
where $C'$ is a constant and the last inequality follows from the fact that $\mathcal{F}[a_{\frac{1}{2\alpha}}^{-1}\varphi_{x_0}(\eta)]$ is a Schwartz function.
The binomial expansion gives that there is a constant $c$ such that
$$(1+|\Re x|+|\Im x|)^{2n+1}\leq c \sum_{|a|+|b|\leq 2n+1}|(\Re x)^a||(\Im x)^b|.$$
Thus we have
\begin{equation*}
\begin{split}
&\sup_{z}\int_{\mathbb{C}^n}|\widetilde{\psi_{x}}(z,w)|dv(w)\\
\leq&\frac{c \sum_{|a|+|b|\leq 2n+1}|(\Re x)^a||(\Im x)^b|}{(1+|\Re x|+|\Im x|)^{2n+1}} \sup_{z}\int_{\mathbb{C}^n}|\langle
(b_{x}\mathcal{H}_{\frac{1}{2\alpha}}g)\ast\mathcal{F}[a_{\frac{1}{2\alpha}}^{-1}\varphi_{x_0}] k_z, k_w \rangle|dv(w)\\
\leq &C_{n,\alpha}\frac{\sum_{|a|+|b|\leq 2n+1}\|\mathcal{H}_{\frac{2}{\alpha}}(|\partial^{a}_{\Re y}\partial^{b}_{\Im y}\mathcal{H}_{\frac{1}{2\alpha}}g(y)|) \|_{\infty}}{(1+|\Re x|+|\Im x|)^{2n+1}}
\end{split}
\end{equation*}
Similarly, we have
\begin{equation*}
\begin{split}
&\sup_{w}\int_{\mathbb{C}^n}|\widetilde{\psi_{x}}(z,w)|dv(z)\\
\leq &C_{n,\alpha}\frac{\sum_{|a|+|b|\leq 2n+1}\|\mathcal{H}_{\frac{2}{\alpha}}(|\partial^{a}_{\Re y}\partial^{b}_{\Im y}\mathcal{H}_{\frac{1}{2\alpha}}g(y)|) \|_{\infty}}{(1+|\Re x|+|\Im x|)^{2n+1}}.
\end{split}
\end{equation*}
Thus we conclude
\begin{align*}
\|T_{g_x} \|&=\|T_{\psi_x}\|\\
&\lesssim C_{n,\alpha}\frac{\sum_{|a|+|b|\leq 2n+1}\|\mathcal{H}_{\frac{2}{\alpha}}(|\partial^{a}_{\Re y}\partial^{b}_{\Im y}\mathcal{H}_{\frac{1}{2\alpha}}g(y)|)\|_{\infty} }{(1+|\Re x|+|\Im x|)^{2n+1}}
\end{align*}
to complete the proof.
\end{proof}
Now  we are ready to present the proof of  the main theorem.
\begin{proof}[Proof of Theorem \ref{Theorem 1.2}]
Since $|\partial^{a}_{\Re y}\partial^{b}_{\Im y}\mathcal{H}_{\frac{1}{2\alpha}}g(y)|dv(y)$ is a Carleson measure, it means
$$\sum_{|a|+|b|\leq 2n+1}\|\mathcal{H}_{\frac{2}{\alpha}}(|\partial^{a}_{\Re y}\partial^{b}_{\Im y}\mathcal{H}_{\frac{1}{2\alpha}}g(y)|)\|_{\infty}<\infty.$$

For  $x\in\Gamma$, let $g_x=(\mathcal{H}_{\frac{1}{2\alpha}}g)\ast\mathcal{F}(\varphi_{x} a_{\frac{1}{2\alpha}}^{-1})$.
By Proposition \ref{proposition 3.2}, we have
$$\|T_{g_x}\|\leq
C_{n,\alpha}\frac{\sum_{|a|+|b|\leq 2n+1}\|\mathcal{H}_{\frac{2}{\alpha}}(|\partial^{a}_{\Re y}\partial^{b}_{\Im y}\mathcal{H}_{\frac{1}{2\alpha}}g(y)|)\|_{\infty} }{(1+|\Re x|+|\Im x|)^{2n+1}}.$$
Since
$$\sum_{x\in \Gamma}\frac{1}{(1+|\Re x|+|\Im x|)^{2n+1}}<C <\infty,$$
 the series $\sum_{x\in \Gamma}T_{g_x}$ of bounded operators converges to a bounded operator $X$ in the operator norm topology and
$$\|X\|\leq C_{n,\alpha}C\sum_{|a|+|b|\leq 2n+1}\|\mathcal{H}_{\frac{2}{\alpha}}(|\partial^{a}_{\Re y}\partial^{b}_{\Im y}\mathcal{H}_{\frac{1}{2\alpha}}g(y)|) \|_{\infty}.$$
We will show that $X$ is the bounded extension of $T_g$. To do so,
by Lemma \ref{Lemma 2.2}, for any $z\in\mathbb{C}^n$, we have
\begin{align*}
\langle T_gk_z,k_z\rangle&=\mathcal{H}_{\frac{1}{\alpha}}g(z)\\
&=\sum_{x\in \Gamma}\mathcal{H}_{\frac{1}{\alpha}}\Big[ (\mathcal{H}_{\frac{1}{2\alpha}}g)\ast\mathcal{F}(\varphi_{x} a_{\frac{1}{2\alpha}}^{-1})\Big](z)\\
&=\sum_{x\in \Gamma}\langle T_{(\mathcal{H}_{\frac{1}{2\alpha}}g)\ast\mathcal{F}(\varphi_{x} a_{\frac{1}{2\alpha}}^{-1})}k_z,k_z \rangle\\
&=\sum_{x\in \Gamma}\langle T_{g_x}k_z,k_z \rangle\\
&=\langle Xk_z,k_z\rangle.
\end{align*}
This implies  $$\langle T_gK_z,K_z\rangle= \langle XK_z,K_z\rangle .$$
Since $\langle T_gK_z,K_w\rangle$ and $\langle XK_z,K_w\rangle$ are both analytic with respect to $w$ and anti-analytic with respect to $z$, we have
$$\langle T_gK_z,K_w\rangle=\langle XK_z,K_w\rangle$$
for any $z,w\in \mathbb{C}^n$. Thus this implies
 $$T_gK_z=XK_z$$ for any $z\in \mathbb{C}^n$. So $X$ is the extension of $T_g$ and hence
$\sum_{x\in \Gamma}T_{g_x}$ converges to $T_g$ in the operator norm topology.

To finish the proof of the main theorem,  we need only show that $T_{g_x}$ has the following integral representation:
for any $t\geq 0,$
$$T_{g_x}=\int_{\mathbb{C}^n}\mathcal{F}(\varphi_{x} a_{\frac{1}{2\alpha}+t}^{-1})(y)T_{\tau_{y}\mathcal{H}_{\frac{1}{2\alpha}+t}g}dv(y)
,$$

To do so, first we show that the map
$$y\rightarrow W_yT_{\mathcal{H}_{\frac{1}{2\alpha}+t}g}W^*_y=T_{\tau_{y}\mathcal{H}_{\frac{1}{2\alpha}+t}g}$$
is uniformly continuous from $\mathbb{C}^n$ to the space of bounded linear operator with respect to the operator norm topology.

Since $|\mathcal{H}_{\frac{1}{2\alpha}}g(y)|dv(y)$ is a Carleson measure,  by Lemma \ref{Lemma 2.1}, we have
$$\mathcal{H}_{\frac{1}{4\alpha}}(\mathcal{H}_{\frac{1}{2\alpha}+t}g)=\gamma_{t+\frac{1}{4\alpha}}\ast \mathcal{H}_{\frac{1}{2\alpha}}g$$
is bounded for any $t\geq0$. Thus (\ref{equation 1.2}) and the norm estimation above to replace $g$ by $\mathcal{H}_{\frac{1}{2\alpha}+t}g$ give that
$$\|T_{\mathcal{H}_{\frac{1}{2\alpha}+t}g}\|\lesssim\sum_{|a|+|b|\leq 2n+1}\|\mathcal{H}_{\frac{2}{\alpha}}(|\partial^{a}_{\Re y}\partial^{b}_{\Im y}\mathcal{H}_{\frac{1}{2\alpha}}\mathcal{H}_{\frac{1}{2\alpha}+t}g(y)|) \|_{\infty}\lesssim \|\mathcal{H}_{\frac{1}{4\alpha}}(\mathcal{H}_{\frac{1}{2\alpha}+t}g)\|_{\infty} <\infty.$$
So  $T_{\mathcal{H}_{\frac{1}{2\alpha}+t}g}$ is bounded.
Direct calculation shows
$$W^*_yW_{y'}=e^{-i\alpha \Im\langle -y,y'\rangle}W_{y'-y}$$
for any $y,y'\in \mathbb{C}^n$. As $W_y$ is a unitary operator, the following equalities hold:
\begin{align*}
\|W_yT_{\mathcal{H}_{\frac{1}{2\alpha}+t}g}W^*_y-W_{y'}T_{\mathcal{H}_{\frac{1}{2\alpha}+t}g}W^*_{y'}\|
&=\|T_{\mathcal{H}_{\frac{1}{2\alpha}+t}g}-W_{y'-y}T_{\mathcal{H}_{\frac{1}{2\alpha}+t}g}W^*_{y'-y}\|\\
&=\|T_{\mathcal{H}_{\frac{1}{2\alpha}+t}g}-T_{\tau_{y'-y}\mathcal{H}_{\frac{1}{2\alpha}+t}g}\|\\
&=\|T_{(\mathcal{H}_{\frac{1}{2\alpha}+t}g)-\tau_{y'-y}(\mathcal{H}_{\frac{1}{2\alpha}+t}g)}\|.
\end{align*}
By (\ref{equation 1.2}) and the norm estimation above, for some $s\in (0,\frac{1}{2\alpha})$, we have
\begin{align*}
& \|T_{(\mathcal{H}_{\frac{1}{2\alpha}+t}g)-\tau_{y'-y}(\mathcal{H}_{\frac{1}{2\alpha}+t}g)}\|\\
\lesssim&\sup_{z}|[(\mathcal{H}_{\frac{1}{2\alpha}+t}g)-\tau_{y'-y}(\mathcal{H}_{\frac{1}{2\alpha}+t}g)]\ast\gamma_s(z)|\\
\leq&\sup_{z}\int_{\mathbb{C}^n}|\mathcal{H}_{\frac{1}{2\alpha}+t}g(w)| |\gamma_s(z-w)-\gamma_s(z-w+y-y')|dv(w)\\
=&(\frac{1}{s\pi})^n\sup_{z} \int_{\mathbb{C}^n}|\mathcal{H}_{\frac{1}{2\alpha}+t}g(w)| |\mathrm{e}^{-\frac{|z-w|^{2}}{s}}-\mathrm{e}^{-\frac{|z-w+y-y'|^{2}}{s}}|dv(w)\\
=&(\frac{1}{s\pi})^n\sup_{z} \int_{\mathbb{C}^n}|\mathcal{H}_{\frac{1}{2\alpha}+t}g(w)\mathrm{e}^{-\frac{|z-w|^{2}}{s}}| |1-\mathrm{e}^{\frac{|z-w|^{2}}{s}-\frac{|z-w+y-y'|^{2}}{s}}|dv(w)\\
=&(\frac{1}{s\pi})^n\sup_{z} \int_{\mathbb{C}^n}|\mathcal{H}_{\frac{1}{2\alpha}+t}g(w)\mathrm{e}^{-\frac{|z-w|^{2}}{s}}|
|1-\mathrm{e}^{-\frac{2\Re\langle z-w,y-y'\rangle+|y-y'|^{2}}{s}}|dv(w)\\
\leq&(\frac{1}{s\pi})^n\sup_{z} \int_{\mathbb{C}^n}|\mathcal{H}_{\frac{1}{2\alpha}+t}g(w)\mathrm{e}^{-\frac{|z-w|^{2}}{s}}|
|\sum_{n=1}^{\infty}\frac{1}{n!}(-\frac{2\Re\langle z-w,y-y'\rangle+|y-y'|^{2}}{s})^n |dv(w)\\
\leq&(\frac{1}{s\pi})^n\sup_{z} \int_{\mathbb{C}^n}|\mathcal{H}_{\frac{1}{2\alpha}+t}g(w)\mathrm{e}^{-\frac{|z-w|^{2}}{s}}|
\sum_{n=1}^{\infty}\frac{1}{n!}(\frac{2| z-w||y-y'|+|y-y'|^{2}}{s})^n dv(w).
\end{align*}
Without loss of generality, we can suppose $|y-y'|\leq 1$. The above estimations give
\begin{align*}
& \|T_{(\mathcal{H}_{\frac{1}{2\alpha}+t}g)-\tau_{y'-y}(\mathcal{H}_{\frac{1}{2\alpha}+t}g)}\|\\
\leq&(\frac{1}{s\pi})^n\sup_{z} \int_{\mathbb{C}^n}|\mathcal{H}_{\frac{1}{2\alpha}+t}g(w)\mathrm{e}^{-\frac{|z-w|^{2}}{s}}|
\sum_{n=1}^{\infty}\frac{1}{n!}(\frac{2| z-w|+|y-y'|}{s})^n|y-y'|^n dv(w)\\
\leq&|y-y'|(\frac{1}{s\pi})^n\sup_{z} \int_{\mathbb{C}^n}|\mathcal{H}_{\frac{1}{2\alpha}+t}g(w)\mathrm{e}^{-\frac{|z-w|^{2}}{s}}|
\sum_{n=1}^{\infty}\frac{1}{n!}(\frac{2| z-w|+1}{s})^n dv(w)\\
\leq&|y-y'|(\frac{1}{s\pi})^n\sup_{z} \int_{\mathbb{C}^n}|\mathcal{H}_{\frac{1}{2\alpha}+t}g(w)|\mathrm{e}^{-\frac{|z-w|^{2}}{s}}
e^{\frac{2| z-w|+1}{s}} dv(w)\\
\leq&|y-y'|(\frac{1}{s\pi})^n\sup_{z} \int_{\mathbb{C}^n}|\mathcal{H}_{\frac{1}{2\alpha}+t}g(w)|\mathrm{e}^{\frac{-|z-w|^{2}+4^2+ \frac{|z-w|^2}{4}+1}{s}} dv(w))\\
=&|y-y'|(\frac{1}{s\pi})^n\| |\mathcal{H}_{\frac{1}{2\alpha}+t}g(w)|\ast f\|_{\infty},
\end{align*}
where $f(w)=\mathrm{e}^{\frac{-\frac{3}{4}|w|^{2}+17}{s}}$ is a Schwartz function.  As Lemma \ref{Lemma 2.1} gives
$$(\frac{1}{s\pi})^n\| |\mathcal{H}_{\frac{1}{2\alpha}+t}g(w)|\ast f\|_{\infty}<\infty,$$
the above estimations imply
$$\|W_yT_{\mathcal{H}_{\frac{1}{2\alpha}+t}g}W^*_y-W_{y'}T_{\mathcal{H}_{\frac{1}{2\alpha}+t}g}W^*_{y'}\|\lesssim |y-y'|.$$
Thus the map
$$y\rightarrow W_yT_{\mathcal{H}_{\frac{1}{2\alpha}+t}g}W^*_y$$
is uniformly continuous from $\mathbb{C}^n$ to the space of bounded linear operator with respect to the norm topology. Since $\mathcal{F}(\varphi_{x} a_{\frac{1}{2\alpha}+t}^{-1})$ is a Schwartz function and
$$\|W_yT_{\mathcal{H}_{\frac{1}{2\alpha}+t}g}W^*_y\|\leq \|T_{\mathcal{H}_{\frac{1}{2\alpha}+t}g}\|,$$
 the integral
$$\int_{\mathbb{C}^n}\mathcal{F}(\varphi_{x} a_{\frac{1}{2\alpha}+t}^{-1})(y)W_yT_{\mathcal{H}_{\frac{1}{2\alpha}+t}g}W^*_ydv(y),$$
converges in the operator norm topology.

To establish the integral representation of $T_{g_x}$, as the Berezin transform is injective, next we calculate the Berezin transform of  $\int_{\mathbb{C}^n}W_yT_{\mathcal{H}_{\frac{1}{2\alpha}+t}g}W^*_y\mathcal{F}(\varphi_{x} a_{\frac{1}{2\alpha}+t}^{-1})(y)dv(y)$ and $T_{g_x}$  to get
\begin{align*}
&\langle[\int_{\mathbb{C}^n}W_yT_{\mathcal{H}_{\frac{1}{2\alpha}+t}g}W^*_y\mathcal{F}(\varphi_{x} a_{\frac{1}{2\alpha}+t}^{-1})(y)dv(y)]k_z, k_z\rangle\\
=&\int_{\mathbb{C}^n}\mathcal{F}(\varphi_{x} a_{\frac{1}{2\alpha}+t}^{-1})(y)\langle [W_yT_{\mathcal{H}_{\frac{1}{2\alpha}+t}g}W^*_y]k_z, k_z\rangle dv(y)\\
=&\int_{\mathbb{C}^n}\mathcal{F}(\varphi_{x} a_{\frac{1}{2\alpha}+t}^{-1})(y)\langle [T_{\tau_{y} \mathcal{H}_{\frac{1}{2\alpha}+t}g}]k_z, k_z\rangle dv(y)\\
=&\int_{\mathbb{C}^n}\mathcal{F}(\varphi_{x} a_{\frac{1}{2\alpha}+t}^{-1})(y)\tau_{y} \mathcal{H}_{\frac{1}{2\alpha}+t}g\ast\gamma_{\frac{1}{\alpha}}(z)dv(y)\\
=& \mathcal{H}_{\frac{1}{2\alpha}+t}g\ast\gamma_{\frac{1}{\alpha}}\ast\mathcal{F}(\varphi_{x} a_{\frac{1}{2\alpha}+t}^{-1})(z)\\
=& \mathcal{H}_{\frac{1}{2\alpha}}g\ast{\gamma_{t}}\ast\gamma_{\frac{1}{\alpha}}\ast\mathcal{F}(\varphi_{x} a_{\frac{1}{2\alpha}+t}^{-1})(z)\\
=& \mathcal{H}_{\frac{1}{2\alpha}}g\ast\mathcal{F}(a_t)\ast\gamma_{\frac{1}{\alpha}}\ast\mathcal{F}(\varphi_{x} a_{\frac{1}{2\alpha}+t}^{-1})(z)\\
=& \mathcal{H}_{\frac{1}{2\alpha}}g\ast\mathcal{F}(\varphi_{x} a_{\frac{1}{2\alpha}}^{-1})\ast\gamma_{\frac{1}{\alpha}}(z)\\
=& \langle T_{g_x}k_z, k_z\rangle
\end{align*}
where the second equality follows from
$$W_yT_{\mathcal{H}_{\frac{1}{2\alpha}+t}g}W^*_y=T_{\tau_{y}\mathcal{H}_{\frac{1}{2\alpha}+t}g};$$
the third equality follows from the relation between the Berezin transform and the heat transform;
the fourth and fifth equalities follow from the semigroup property of the heat transform. As the Berezin transform is injective,  we conclude
\begin{align*}T_{g_x}&=\int_{\mathbb{C}^n}\mathcal{F}(\varphi_{x} a_{\frac{1}{2\alpha}+t}^{-1})(y)W_yT_{\mathcal{H}_{\frac{1}{2\alpha}+t}g}W^*_ydv(y)\\
&=\int_{\mathbb{C}^n}\mathcal{F}(\varphi_{x} a_{\frac{1}{2\alpha}+t}^{-1})(y)T_{\tau_{y}\mathcal{H}_{\frac{1}{2\alpha}+t}g}dv(y)
 \end{align*}
to complete the proof.
\end{proof}

\section{Schatten $p$-class}
In this section, we will apply our decomposition theory for a Toeplitz operator to estimate the Schatten $p$-norm of the product of two Toeplitz operators.

Let $\mathcal{S}_p$ denote the Schatten $p$-class on $F_{\alpha}^2$. A compact operator $A$ on $F^2_{\alpha}$ is in $\mathcal{S}_p$ if
$$\|A\|_{\mathcal{S}_p}=\sup\Big\{\Big(\sum\|Ae_n\|^p\Big)^{1/p}: \{e_n\} \text{ is an orthonormal set }\Big\}<\infty.$$
For a Toeplitz operator $T_f$, we have
\begin{equation}\label{schatten}
\|T_f\|_{\mathcal{S}_p}\lesssim \|f\|_{L^p(\mathbb{C}^n,dv)},
\end{equation}
see \cite[Lemma 6.30]{Zhu}.

\begin{lemma}\label{weakint}
Let $\{X,M,m\}$ be a measure space and $\mathcal{H}$ be a separable Hilbert space. Let $\mathcal{S}_p$  denote the Schatten $p$-class on $\mathcal{H}$. Suppose that $F: X\rightarrow\mathcal{S}_p$ is a weakly M-measurable map.
If
$$\int_X \|F(x)\|_{\mathcal{S}_p}dm(x)<\infty,$$
then
$$K=\int_{X}F(x)dm(x)\in \mathcal{S}_p,\text{ and } \|K\|_{\mathcal{S}_p}\leq \int_X \|F(x)\|_{\mathcal{S}_p}dm(x)$$
where the integral is taken in the weak sense.
\end{lemma}
\begin{proof}Let $\{e_i:i=1,\cdots,n,\cdots\}$ denote the orthogonal basis in $\mathcal{H}$. Let $P_n$ be the project on to the space generated by
$\{e_1,\cdots,e_n\}$. We have
$$K-KP_n=\int_{X}F(x)-F(x)P_ndm(x).$$
By the definition of the $\mathcal{S}_p$ norm, we have
$$\|K-KP_n\|_{\mathcal{S}_p}\leq\int_{X}\|F(x)-F(x)P_n\|_{\mathcal{S}_p}dm(x)\text{ and }\lim_{n\rightarrow\infty}\|F(x)-F(x)P_n\|_{\mathcal{S}_p}=0.$$
Since
$$\|F(x)-F(x)P_n\|_{\mathcal{S}_p}\leqslant 2\|F(x)\|_{\mathcal{S}_p},$$
which is integrable, by the dominated convergence theorem we have
$$\lim_{n\rightarrow\infty}\|K-KP_n\|_{\mathcal{S}_p}=\int_{X}\lim_{n\rightarrow\infty}\|F(x)-F(x)P_n\|_{\mathcal{S}_p}dm(x)=0.$$
Since $\mathcal{S}_p$ is closed with respect to the Schatten $p$-norm, we have $K\in \mathcal{S}_p.$ The norm estimation follows directly.
\end{proof}

Recall that the Weyl operator $W_z$ on $F_{\alpha}^2$ is defined by
$$ W_zf(w)=k_zf(w-z).$$
One can check that
\begin{equation}\label{weyl}
W_wW_z=e^{-i \frac{\Im(w\cdot \overline{z})}{t}}W_{w+z},\quad W_wK_z=e^{-i \frac{\Im(w\cdot \overline{z})}{t}}k_{w+z} \text{ and } \|W_zf\|_{F_{\alpha}^2}=\|f\|_{F_{\alpha}^2}.
\end{equation}

\begin{theorem}\label{Theorem 4.4}
Let $A$ be a bounded operator on $F_{\alpha}^2$, if
$$\int_{\mathbb{C}^n}\Big(\int_{\mathbb{C}^n}|\langle Ak_z,k_{z+w}\rangle|^pdv(z)\Big)^{1/p} dv(w)< \infty,$$
then $A$ is in the Schatten $p$-class and
$$ \|A\|_{\mathcal{S}_p}\leq \int_{\mathbb{C}^n}\Big(\int_{\mathbb{C}^n}|\langle Ak_z,k_{z+w}\rangle|^pdv(z)\Big)^{1/p} dv(w).$$
\end{theorem}
\begin{proof}
For any $f,g\in F_{\alpha}^2$ we have
\begin{align*}
\langle Af,g\rangle&=\frac{\alpha^n}{\pi^n}\int_{\mathbb{C}^n}\langle Af,k_w\rangle \langle k_w,g\rangle dv(w)\\
&=\frac{\alpha^n}{\pi^n}\int_{\mathbb{C}^n}\langle f,A^* k_w\rangle \langle k_w,g\rangle dv(w)\\
&=\frac{\alpha^{2n}}{\pi^{2n}}\int_{\mathbb{C}^n}\int_{\mathbb{C}^n}\langle f,k_z\rangle \langle k_z ,A^* k_w\rangle \langle k_w,g\rangle dv(z)dv(w)\\
&=\frac{\alpha^{2n}}{\pi^{2n}}\int_{\mathbb{C}^n}\int_{\mathbb{C}^n}\langle f,k_z\rangle \langle Ak_z , k_{w+z}\rangle \langle k_{w+z},g\rangle dv(w)dv(z)\\
&=\frac{\alpha^{2n}}{\pi^{2n}}\int_{\mathbb{C}^n}\int_{\mathbb{C}^n}\langle f,k_z\rangle \langle Ak_z , W_wk_{z}\rangle \langle W_wk_{z},g\rangle dv(w)dv(z)\\
&=\frac{\alpha^{2n}}{\pi^{2n}}\int_{\mathbb{C}^n}\int_{\mathbb{C}^n}\langle f,k_z\rangle \langle Ak_z , W_wk_{z}\rangle \langle k_{z},W^*_wg\rangle dv(w)dv(z).
\end{align*}
Let
$$S_w=\frac{\alpha^n}{\pi^n}\int_{\mathbb{C}^n}\langle Ak_z,W_wk_{z}\rangle k_{z}\otimes k_z dv(z),$$
where the integral is taken in the weak sense. $S_w$ is actually a Toeplitz operator, we have
\begin{align*}
\langle Af,g\rangle=\frac{\alpha^n}{\pi^n}\int_{\mathbb{C}^n} \langle S_w f,W^*_w g\rangle dv(w)
=\frac{\alpha^n}{\pi^n}\int_{\mathbb{C}^n} \langle W_w S_w f, g\rangle dv(w).
\end{align*}
Thus
$$A=\frac{\alpha^n}{\pi^n}\int_{\mathbb{C}^n}W_w S_wdv(w),$$
where the integral is taken in the weak sense. Since $S_w$ is a Toeplitz operator, we have
$$\|W_w S_w\|_{\mathcal{S}_p}\leq\|S_w\|_{\mathcal{S}_p}\leq \Big(\int_{\mathbb{C}^n}|\langle Ak_z,W_w k_{z}\rangle|^pdv(z)\Big)^{1/p}=\Big(\int_{\mathbb{C}^n}|\langle Ak_z, k_{z+w}\rangle|^pdv(z)\Big)^{1/p}.$$
By the hypothesis and Lemma \ref{weakint}, we obtain $A\in\mathcal{S}_p $. The norm estimation follows directly.
\end{proof}

For any $a,b\in \mathbb{Z}_{+}^n$, let
$$J^{a,b}g(y)=\partial^{a}_{\Re y}\partial^{b}_{\Im y}\mathcal{H}_{\frac{1}{2\alpha}}g(y).$$

In the proof of \cite[Theorem 1]{Coburn2010}, the authors get the Schatten $p$-norm estimation for one Toeplitz operator through the result about the pseudo-differenial operator in \cite{CR}.
That is
$$\|T_g\|_{\mathcal{S}_p}\lesssim \sum_{|a|+|b|\leq 2n+1 }  \Big(\int_{\mathbb{C}^n}|J^{a,b}g(\eta)|^pdv(\eta)\Big)^{\frac{1}{p}}.$$
We will show an estimation for the product of Toeplitz operators and our proof has nothing to do with the pseudo-differenial operator. Moreover, we will show that our result implies their result.

\begin{theorem}\label{product}
Let $g$ and $f$ be two measurable functions on $\mathbb{C}^n$ such that $gk_w, fk_w\in L^2(\mathbb{C}^n,d\lambda_{\alpha})$ for any $w\in \mathbb{C}^n$.
If for any $a,b\in \mathbb{Z}_{+}^n$ with $|a|+|b|\leq 2n+1$, $|J^{a,b}g(y)|dv(y)$ and $|J^{a,b}f(y)|dv(y)$ are Carleson measures on the Fock space. we have
$$\|T_fT_g\|_{\mathcal{S}_p}\lesssim \sum_{\mbox{\tiny$\begin{array}{c}
  |a|+|b|\leq 2n+1\\
  |a'|+|b'|\leq 2n+1\end{array}$} }
\sup_{w\in \mathbb{C}^n}\Big(\int_{\mathbb{C}^n} \int_{\mathbb{C}^n}|J^{a,b}f(\xi)|^p|J^{a',b'}g(\eta)|^pe^{-\frac{\alpha}{2}|\xi-\eta+w|^2}dv(\xi)dv(\eta)\Big)^{1/p}.$$
\end{theorem}
\begin{proof}
By Theorem \ref{Theorem 1.2}, we know that $T_f$ and $T_g$ are bounded and we have two decompositions
$$T_f=\sum_{y\in\Gamma}T_{f_y}\text{\quad and \quad}T_g=\sum_{x\in\Gamma}T_{g_x},$$
where $f_y=(\mathcal{H}_{\frac{1}{2\alpha}}f)\ast\mathcal{F}(\varphi_{y} a_{\frac{1}{2\alpha}}^{-1})$ and $g_x=(\mathcal{H}_{\frac{1}{2\alpha}}g)\ast\mathcal{F}(\varphi_{x} a_{\frac{1}{2\alpha}}^{-1})$.
Thus we have
\begin{align*}
\|T_fT_g\|_{\mathcal{S}_p}\leq & \|\sum_{y\in\Gamma}T_{f_y}\sum_{x\in\Gamma}T_{g_x}\|_{\mathcal{S}_p}\leq \sum_{y\in\Gamma}\sum_{x\in\Gamma}\|T_{f_y}T_{g_x}\|_{\mathcal{S}_p}\\
= & \sum_{y\in\Gamma}\frac{(1+|\Re y|+|\Im y|)^{2n+1}}{(1+|\Re y|+|\Im y|)^{2n+1}}\sum_{x\in\Gamma}\frac{(1+|\Re x|+|\Im x|)^{2n+1}}{(1+|\Re x|+|\Im x|)^{2n+1}} \|T_{f_y}T_{g_x}\|_{\mathcal{S}_p}\\
\lesssim & \sup_{x}\sup_{y} \sum_{\mbox{\tiny$\begin{array}{c}
  |a|+|b|\leq 2n+1\\
  |a'|+|b'|\leq 2n+1\end{array}$} } |(\Re y)^{a}||(\Im y)^{b}||(\Re x)^{a'}||(\Im x)^{b'}| \|T_{f_y}T_{g_x}\|_{\mathcal{S}_p}
\end{align*}
By Lemma \ref{Lemma 2.3} we have 
$$W_{\frac{-i\pi y}{\alpha}}T_{f_y}= T_{(b_{y}\mathcal{H}_{\frac{1}{2\alpha}}f)\ast \tau_{\frac{i\pi y}{2\alpha} }\mathcal{F}[a_{\frac{1}{2\alpha}}^{-1}\varphi_{y_0}]}\text{\quad and\quad}T_{g_x}W_{-\frac{i\pi x}{\alpha}}=  T_{(b_{x}\mathcal{H}_{\frac{1}{2\alpha}}g)\ast \tau_{-\frac{i\pi x}{2\alpha} }\mathcal{F}[a_{\frac{1}{2\alpha}}^{-1}\varphi_{x_0}]},$$
 where $y_0=0$. 
Then we get
$$ \|T_{f_y}T_{g_x}\|_{\mathcal{S}_p}=\|W_{\frac{-i\pi y}{\alpha}}T_{f_y}T_{g_x}W_{-\frac{i\pi x}{\alpha}}\|_{\mathcal{S}_p}
=\|T_{(b_{y}\mathcal{H}_{\frac{1}{2\alpha}}f)\ast \tau_{\frac{i\pi y}{2\alpha} }\mathcal{F}[a_{\frac{1}{2\alpha}}^{-1}\varphi_{y_0}]} T_{(b_{x}\mathcal{H}_{\frac{1}{2\alpha}}g)\ast \tau_{-\frac{i\pi x}{2\alpha} }\mathcal{F}[a_{\frac{1}{2\alpha}}^{-1}\varphi_{x_0}]}\|_{\mathcal{S}_p}.$$
By (\ref{equation 3.1}), we have
$$
{(\Re y)^{a}(\Im y)^{b}(b_{y}\mathcal{H}_{\frac{1}{2\alpha}}f)\ast \tau_{\frac{i\pi y}{2\alpha} }\mathcal{F}[a_{\frac{1}{2\alpha}}^{-1}\varphi_{y_0}]}=\sum_{c<a, d<b} (b_{y}J^{c,d}f)\ast h_{y,c,d}
$$
and
$$
{(\Re x)^{a'}(\Im x)^{b'}(b_{x}\mathcal{H}_{\frac{1}{2\alpha}}g)\ast \tau_{\frac{-i\pi x}{2\alpha} }\mathcal{F}[a_{\frac{1}{2\alpha}}^{-1}\varphi_{x_0}]}=\sum_{c<a', d<b'} (b_{x}J^{c,d}f)\ast h_{x,c,d}
$$
where $h_{y,c,d}$ and $h_{x,c,d}$ are Schwartz functions and
$$\int_{\mathbb{C}_{n}}|h_{x,c,d}(z)|dv(z)\text{ and }\int_{\mathbb{C}_{n}}|h_{y,c,d}(z)|dv(z)$$
are independent of $x$ and $y$.
Thus
\begin{align*}
\sum_{\mbox{\tiny$\begin{array}{c}
|a|+|b|\leq 2n+1\\
|a'|+|b'|\leq 2n+1\end{array}$} }& |(\Re y)^{a}||(\Im y)^{b}||(\Re x)^{a'}||(\Im x)^{b'}| \|T_{f_y}T_{g_x}\|_{\mathcal{S}_p}\\
\lesssim \sum_{\mbox{\tiny$\begin{array}{c}
|a|+|b|\leq 2n+1\\
|a'|+|b'|\leq 2n+1\end{array}$} }&\|T_{(b_{y}J^{a,b}f)\ast h_{y,a,b}} T_{(b_{x}J^{a',b'}g)\ast h_{x,a',b'}}\|_{\mathcal{S}_p}.
\end{align*}
Thus, we have
$$\|T_fT_g\|_{\mathcal{S}_p}\lesssim \sum_{\mbox{\tiny$\begin{array}{c}
|a|+|b|\leq 2n+1\\
|a'|+|b'|\leq 2n+1\end{array}$} } \|T_{(b_{y}J^{a,b}f)\ast h_{y,a,b}} T_{(b_{x}J^{a',b'}g)\ast h_{x,a',b'}}\|_{\mathcal{S}_p}.
$$

By Lemma \ref{Lemma 2.1}, we have $T_{(b_{y}J^{a,b}f)\ast h_{y,a,b}} T_{(b_{x}J^{a',b'}g)\ast h_{x,a',b'}}$ is a product of Toeplitz operators with bounded symbols, we can apply Theorem \ref{Theorem 4.4}. We have
\begin{align*}
&|\langle T_{(b_{y}J^{a,b}f)\ast h_{y,a,b}} T_{(b_{x}J^{a',b'}g)\ast h_{x,a',b'}}k_z,k_{z+w}\rangle|\\
\leq &\Big|\int_{\mathbb{C}^n}J^{a,b}f \ast h_{y,a,b}(\xi) \int_{\mathbb{C}^n} J^{a',b'}g \ast h_{x,a',b'}(\eta)k_z(\eta) \overline{K_{\xi}(\eta)}d\lambda_{\alpha}(\eta)
\overline{k_{z+w}(\xi)}d\lambda_{\alpha}(\eta) \Big|\\
\leq &\int_{\mathbb{C}^n}\int_{\mathbb{C}^n} |J^{a,b}f|\ast |h_{y,a,b}|(\xi)|J^{a',b'}g|\ast |h_{x,a',b'}|(\eta)
e^{\frac{\alpha|z-\eta|^2}{2}}e^{\frac{\alpha|\xi-\eta|^2}{2}}e^{\frac{\alpha|\xi-z-w|^2}{2}}dv(\xi) dv(\eta).
\end{align*}
Denote
$$L^{a,b}_{y}f(\xi)=|J^{a,b}f|\ast |h_{y,a,b}|(\xi) \text{ and } L^{a',b'}_{x}g(\eta)=|J^{a',b'}g|\ast |h_{x,a',b'}|(\eta).$$
By Theorem \ref{Theorem 4.4}, we have
\begin{align*}
&\|T_{(b_{y}J^{a,b}f)\ast h_{y,a,b}} T_{(b_{x}J^{a',b'}g)\ast h_{x,a',b'}}\|_{\mathcal{S}_p}\\
\leq & \int_{\mathbb{C}^n}\Big(\int_{\mathbb{C}^n}\big|\int_{\mathbb{C}^n}\int_{\mathbb{C}^n}
L^{a,b}_{y}f(\xi)L^{a',b'}_{x}g(\eta)e^{\frac{\alpha|z-\eta|^2}{2}}e^{\frac{\alpha|\xi-\eta|^2}{2}}e^{\frac{\alpha|\xi-z-w|^2}{2}}dv(\xi) dv(\eta)\big|^pdv(z)\Big)^{\frac{1}{p}}dv(w).
\end{align*}
For simplicity, we will denote $dv(w)$ by $dw$. If $p=1$, we have
\begin{align*}
&\|T_{(b_{y}J^{a,b}f)\ast h_{y,a,b}} T_{(b_{x}J^{a',b'}g)\ast h_{x,a',b'}}\|_{\mathcal{S}_p}\\
\leq& \int_{\mathbb{C}^n}\int_{\mathbb{C}^n}\int_{\mathbb{C}^n}\int_{\mathbb{C}^n}L^{a,b}_{y}f(\xi)L^{a',b'}_{x}g(\eta)
e^{\frac{\alpha|z-\eta|^2}{2}}e^{\frac{\alpha|\xi-\eta|^2}{2}}e^{\frac{\alpha|\xi-z-w|^2}{2}}d\xi d\eta dzdw\\
\leq& \int_{\mathbb{C}^n}\int_{\mathbb{C}^n}\int_{\mathbb{C}^n}\int_{\mathbb{C}^n}L^{a,b}_{y}f(\xi)L^{a',b'}_{x}g(\eta)
e^{\frac{\alpha|z-\eta|^2}{2}}e^{\frac{\alpha|\xi-\eta|^2}{2}}e^{\frac{\alpha|\xi-z-w|^2}{2}}dw dz d\xi d\eta \\
\lesssim &\int_{\mathbb{C}^n}\int_{\mathbb{C}^n}L^{a,b}_{y}f(\xi)L^{a',b'}_{x}g(\eta)e^{\frac{\alpha|\xi-\eta|^2}{2}}d\xi d\eta\\
\lesssim &\int_{\mathbb{C}^n}\int_{\mathbb{C}^n} |J^{a,b}f|\ast |h_{y,a,b}|(\xi) |J^{a',b'}g|\ast |h_{y,a,b}|(\eta) e^{\frac{\alpha|\xi-\eta|^2}{2}}d\xi d\eta\\
\lesssim &\int_{\mathbb{C}^n}\int_{\mathbb{C}^n} \int_{\mathbb{C}^n}\int_{\mathbb{C}^n} |J^{a,b}f(\xi-s)| |J^{a',b'}g(\eta-r)| e^{\frac{\alpha|\xi-\eta|^2}{2}}d\xi d\eta |h_{y,a,b}(s)|ds |h_{x,a,b}(r)|dr \\
\lesssim & \sup_{s\in \mathbb{C}^n}\int_{\mathbb{C}^n} \int_{\mathbb{C}^n}|J^{a,b}f(\xi)|^p|J^{a',b'}g(\eta)|^pe^{-\frac{\alpha}{2}|\xi-\eta+s|^2}d\xi d\eta.
\end{align*}
We have completed the proof when $p=1$.
If $p>1$, let $q>1$ such that $1/p+1/q=1$. We have
\begin{align*}
&\|T_{(b_{y}J^{a,b}f)\ast h_{y,a,b}} T_{(b_{x}J^{a',b'}g)\ast h_{x,a',b'}}\|_{\mathcal{S}_p}\\
\leq & \int_{\mathbb{C}^n}\Big(\int_{\mathbb{C}^n}\big|\int_{\mathbb{C}^n}\int_{\mathbb{C}^n}
L^{a,b}_{y}f(\xi)L^{a',b'}_{x}g(\eta)e^{\frac{\alpha|z-\eta|^2}{2}}e^{\frac{\alpha|\xi-\eta|^2}{2}}e^{\frac{\alpha|\xi-z-w|^2}{2}}d\xi d\eta\big|^pdz\Big)^{\frac{1}{p}}dw\\
\leq & \int_{\mathbb{C}^n}\Big(\int_{\mathbb{C}^n}\big|\int_{\mathbb{C}^n}\int_{\mathbb{C}^n}
L^{a,b}_{y}f(\xi)L^{a',b'}_{x}g(\eta)e^{\frac{\alpha|z-\eta|^2}{2p}} e^{\frac{\alpha|\xi-\eta|^2}{2p}} e^{\frac{\alpha|z-\eta|^2}{2q}} e^{\frac{\alpha|\xi-\eta|^2}{2q}} e^{\frac{\alpha|\xi-z-w|^2}{2}}d\xi d\eta\big|^pdz\Big)^{\frac{1}{p}}dw
\end{align*}
Using H$\ddot{\mathrm{o}}$lder inequality, we get
\begin{align*}
&\big|\int_{\mathbb{C}^n}\int_{\mathbb{C}^n}
L^{a,b}_{y}f(\xi)L^{a',b'}_{x}g(\eta)e^{\frac{\alpha|z-\eta|^2}{2p}} e^{\frac{\alpha|\xi-\eta|^2}{2p}} e^{\frac{\alpha|z-\eta|^2}{2q}} e^{\frac{\alpha|\xi-\eta|^2}{2q}} e^{\frac{\alpha|\xi-z-w|^2}{2}}d\xi d\eta\big|^p\\
\leq & \int_{\mathbb{C}^n}\int_{\mathbb{C}^n}
|L^{a,b}_{y}f(\xi)L^{a',b'}_{x}g(\eta)|^p e^{\frac{\alpha|z-\eta|^2}{2}} e^{\frac{\alpha|\xi-\eta|^2}{2}}d\xi d\eta
\Big(\int_{\mathbb{C}^n}\int_{\mathbb{C}^n}e^{\frac{\alpha|z-\eta|^2}{2}} e^{\frac{\alpha|\xi-\eta|^2}{2}}
e^{\frac{q\alpha|\xi-z-w|^2}{2}}d\xi d\eta\Big)^{\frac{p}{q}}\\
\lesssim & \int_{\mathbb{C}^n}\int_{\mathbb{C}^n}
|L^{a,b}_{y}f(\xi)L^{a',b'}_{x}g(\eta)|^p e^{\frac{\alpha|z-\eta|^2}{2}} e^{\frac{\alpha|\xi-\eta|^2}{2}}d\xi d\eta
\Big(\int_{\mathbb{C}^n} e^{\frac{\alpha|\xi-z|^2}{4}} e^{\frac{q\alpha|\xi-z-w|^2}{2}} d\xi\Big)^{\frac{p}{q}}\\
\lesssim& \int_{\mathbb{C}^n}\int_{\mathbb{C}^n}
|L^{a,b}_{y}f(\xi)L^{a',b'}_{x}g(\eta)|^p e^{\frac{\alpha|z-\eta|^2}{2}} e^{\frac{\alpha|\xi-\eta|^2}{2}}d\xi d\eta
e^{-\frac{p\alpha}{2+4q}|w|^2}.
\end{align*}
Thus

\begin{align*}
&\|T_{(b_{y}J^{a,b}f)\ast h_{y,a,b}} T_{(b_{x}J^{a',b'}g)\ast h_{x,a',b'}}\|_{\mathcal{S}_p}\\
\lesssim& \int_{\mathbb{C}^n} \Big(\int_{\mathbb{C}^n}\int_{\mathbb{C}^n}\int_{\mathbb{C}^n}
|L^{a,b}_{y}f(\xi)L^{a',b'}_{x}g(\eta)|^p e^{\frac{\alpha|z-\eta|^2}{2}} e^{\frac{\alpha|\xi-\eta|^2}{2}}d\xi d\eta
e^{-\frac{p\alpha}{2+4q}|w|^2}dz\Big)^{\frac{1}{p}}dw\\
\lesssim& \Big(\int_{\mathbb{C}^n}\int_{\mathbb{C}^n}
|L^{a,b}_{y}f(\xi)L^{a',b'}_{x}g(\eta)|^p e^{\frac{\alpha|\xi-\eta|^2}{2}}d\xi d\eta \Big)^{\frac{1}{p}}
\end{align*}
\begin{align*}
\lesssim& \Big(\int_{\mathbb{C}^n}\int_{\mathbb{C}^n}
|J^{a,b}f|\ast |h_{y,a,b}|(\xi)|J^{a',b'}g|\ast |h_{x,a',b'}|(\eta)|^p e^{\frac{\alpha|\xi-\eta|^2}{2}}d\xi d\eta \Big)^{\frac{1}{p}}\\
\lesssim& \int_{\mathbb{C}^n}\int_{\mathbb{C}^n}\Big(\int_{\mathbb{C}^n}\int_{\mathbb{C}^n}
|J^{a,b}f(\xi-s)|^p |J^{a',b'}g(\eta-r)|^p e^{\frac{\alpha|\xi-\eta|^2}{2}}d\xi d\eta \Big)^{\frac{1}{p}} |h_{y,a,b}(s)| |h_{x,a',b'}(r)|dsdr\\
\lesssim& \sup_{w\in \mathbb{C}^n}\Big(\int_{\mathbb{C}^n} \int_{\mathbb{C}^n}|J^{a,b}f(\xi)|^p|J^{a',b'}g(\eta)|^pe^{-\frac{\alpha}{2}|\xi-\eta+w|^2}dv(\xi)dv(\eta)\Big)^{\frac{1}{p}}.
\end{align*}
We have completed the proof.
\end{proof}

\begin{corollary}Let $g$ be a measurable function on $\mathbb{C}^n$ such that $gk_w \in L^2(\mathbb{C}^n,d\lambda_{\alpha})$ for any $w\in \mathbb{C}^n$.
We have
$$\|T_g\|_{\mathcal{S}_p}\lesssim \sum_{|a'|+|b'|\leq 2n+1 }  \Big(\int_{\mathbb{C}^n}|J^{a',b'}g(\eta)|^pdv(\eta)\Big)^{\frac{1}{p}}.$$
\end{corollary}
\begin{proof}
If $f=1$, then $T_fT_g=T_g$. By Theorem \ref{product}, we have
\begin{align*}
\|T_g\|_{\mathcal{S}_p}=&\|T_fT_g\|_{\mathcal{S}_p}\lesssim \sum_{|a'|+|b'|\leq 2n+1 }
\sup_{w\in \mathbb{C}^n}\Big(\int_{\mathbb{C}^n} \int_{\mathbb{C}^n}|J^{a',b'}g(\eta)|^pe^{-\frac{\alpha}{2}|\xi-\eta+w|^2}dv(\xi)dv(\eta)\Big)^{\frac{1}{p}}\\
\leq & \sum_{|a'|+|b'|\leq 2n+1 }  \Big(\int_{\mathbb{C}^n}|J^{a',b'}g(\eta)|^pdv(\eta)\Big)^{\frac{1}{p}}.
\end{align*}
\end{proof}

\section*{Acknowledgement}
It is our pleasure to thank Robert Fulsche for useful comments.


\bibliographystyle{plain} 

\end{document}